\documentclass[11pt, leqno]{amsart}

\usepackage[square, numbers, comma]{natbib} 

\usepackage[usenames,dvipsnames,svgnames,table]{xcolor}
\usepackage{amsmath,amsthm,amssymb,amssymb,esint,verbatim,tabularx,graphicx}
\usepackage{enumerate}
\usepackage{amsmath}
\usepackage{fancyhdr}
\usepackage{pgf,tikz}
\usetikzlibrary{arrows}

\usepackage[utf8]{inputenc}
\usepackage{color}
\usepackage{hyperref}
\usepackage{pdfpages}
\usepackage{mathtools}
\usepackage{subfigure}

\hypersetup{urlcolor=blue, colorlinks=true} 

\newtheorem{theorem}{Theorem}[section]
\newtheorem{proposition}[theorem]{Proposition}
\newtheorem{lemma}[theorem]{Lemma}
\newtheorem{corollary}[theorem]{Corollary}

\theoremstyle{definition}

\newtheorem{definition}[theorem]{Definition}
\newtheorem{remark}[theorem]{Remark}
\newtheorem{example}[theorem]{Example}

\newcommand{\R}{\ensuremath{\mathbb{R}}}
\newcommand{\N}{\ensuremath{\mathbb{N}}}
\newcommand{\Z}{\ensuremath{\mathbb{Z}}}
\newcommand{\Q}{\mathbb{Q}}

\newcommand{\Levy}{\ensuremath{\mathcal{L}}}
\newcommand{\dd}{\,\mathrm{d}}
\newcommand{\limsu}{\varlimsup}

\newcommand{\q}{\mathcal{Q}}
\newcommand{\supsupp}{\sup_{b\in  \supp (\mu) }}

\DeclareMathOperator{\supp}{supp}
\DeclarePairedDelimiter\floor{\lfloor}{\rfloor}

\begin{document}

\title[Nonlocal operators and the Liouville theorem]{Characterization of nonlocal diffusion operators 
satisfying the Liouville theorem.
 Irrational numbers and subgroups of $\R^d$}

\author[N.~Alibaud]{Natha\"el Alibaud}
\address[N.~Alibaud]{ENSMM\\
26 Chemin de l'Epitaphe\\ 25030 Besan\c{c}on cedex\\ Fran\-ce
and\\
LMB\\ UMR CNRS 6623\\ Universit\'e de Bourgogne Franche-Comt\'e (UBFC)\\ France}
\email{nathael.alibaud\@@{}ens2m.fr}


\author[F.~del Teso]{F\'elix del Teso}
\address[F.~del Teso]{Department of Mathematical Sciences\\
Norwegian University of Science and Technology (NTNU)\\
N-7491 Trondheim, Norway} 
\email{felix.delteso@ntnu.no}

\author[J. Endal]{J\o rgen Endal}
\address[J. Endal]{Department of Mathematical Sciences\\
Norwegian University of Science and Technology (NTNU)\\
N-7491 Trondheim, Norway} 
\email{jorgen.endal\@@{}ntnu.no}

\author[E.~R.~Jakobsen]{Espen R. Jakobsen}
\address[E.~R.~Jakobsen]{Department of Mathematics\\
Norwegian University of Science and Technology (NTNU)\\
N-7491 Trondheim, Norway} 
\email{espen.jakobsen\@@{}ntnu.no}




\begin{abstract}
\noindent  We investigate the characterization of generators $\mathcal{L}$ of L\'evy processes satisfying the Liouville theorem: Bounded functions $u$ solving $\mathcal{L}[u]=0$ are constant. These operators are degenerate elliptic of the form $\mathcal{L}=\mathcal{L}^{\sigma,b}+\Levy^\mu$ for some  local part $
\mathcal{L}^{\sigma,b}[u]=\textup{tr}(\sigma \sigma^{\texttt{T}} D^2u)+b \cdot Du
$
and nonlocal part 
$$
\Levy^\mu[u](x)=\int \big(u(x+z)-u(x)-z \cdot Du(x) \mathbf{1}_{|z| \leq 1}\big) \dd \mu(z),
$$ 
where $\mu \geq 0$ is a so-called L\'evy measure possibly unbounded for small $z$. In this paper, we focus on the pure nonlocal case $\sigma=0$ and $b=0$, where we assume in addition that $\mu$ is symmetric which corresponds to self-adjoint pure jump L\'evy operators $\mathcal{L}=\Levy^\mu$. The case of general L\'evy operators $\mathcal{L}=\mathcal{L}^{\sigma,b}+\Levy^\mu$ will be considered in the forthcoming paper \cite{AlDTEnJa18}. In our setting, we  show that $\Levy^\mu[u]=0$ if and only if $u$ is periodic wrt the subgroup generated by the support of $\mu$. Therefore, the Liouville property holds if and only if this subgroup is dense, and in space dimension $d=1$ there is an equivalent condition in terms of  irrational numbers. In dimension $d \geq 1$, we have a clearer view of the operators \textit{not} satisfying the Liouville theorem whose general form is precisely identified.
The proofs are based on arguments of propagation of maximum. 
\end{abstract}

\maketitle 
\tableofcontents

\section{Introduction}

The classical Liouville theorem states that
bounded functions $u$ solving $\Delta u=0$ in the distributional sense are constant.
 We revisit this result for   nonlocal (or anomalous) diffusion operators  of the form 
\begin{equation}\label{def:levy}
\Levy ^\mu [ u ](x):={ \textup{P.V.} } \int_{\R^{d}  \setminus \{0\}} \big( u (x+z)- u  (x)\big) \dd\mu(z),
\end{equation}
 where $\mu \geq 0$ is symmetric and $
\int (|z|^2\wedge1)\dd\mu(z)<+\infty.$ 

\medskip

Our motivation goes back to the result of Courr\`ege \cite{Cou64} on linear operators satisfying the positive maximum principle. These operators can have an integro-differential part which is precisely $\Levy^\mu$. Our symmetry assumption on $\mu$ means that we are considering self-adjoint operators. Operators of the form \eqref{def:levy} coincides with generators of    symmetric pure-jump L\'evy processes \cite{Sat99,App09}. A famous example is the fractional Laplacian $-(-\Delta)^\frac{\alpha}{2}$ which corresponds to   $\alpha$-stable processes. Other examples from biology, mechanics, etc., are
convolution operators 
$
J\ast u-u
$
(cf. \cite{Cov08,BrCoHaVa17, BrCo18}),
relativistic Schr\"odinger operators $m^\alpha I-(m^2I-\Delta)^\frac{\alpha}{2}$ (cf. \cite{FaFe15,FaWe16}), or finite difference  discretization like   $(u(x+h)+u(x-h)-2u(x))/h^2 \approx \partial_{x}^2 u(x)$, as well as other general similar discretizations  in $\R^d$ for  local and nonlocal diffusion operators (cf. \cite{DTEnJa17b,DTEnJa17a,DTEnJa18a,DTEnJa18b}). 

\medskip

 There is a huge literature on the Liouville theorem and we focus our attention on nonlocal PDEs. Such a result is more or less understood for the fractional Laplacian or variants \cite{Lan72,BoKuNo02, StZh13,CaSi14, ZhChCuYu14, ChDALi15, Fal15,BaDPG-MQu17}, certain L\'evy operators \cite{BaBaGu00,PrZa04,Ser15b, Ser15a, R-OSe16a, R-OSe16b,DTEnJa17b,DTEnJa17a}, relativistic Schr\"{o}dinger operators \cite{FaWe16}, or convolution operators \cite{BrCoHaVa17, BrCo18}. The techniques vary from Fourier analysis, potential theory, probabilistic methods, or classical PDEs arguments. Let us also refer to works on propagation of maximum which is a closely related topic (cf. \cite{Cov08,Cio12,ShHa16,DTEnJa17b,DTEnJa17a,BrCoHaVa17,HuDuWu18,BrCo18}). 

\medskip

As far we know, there is not yet any complete classification of  nonlocal  operators for which the Liouville property holds.  This is precisely what we provide in  this paper  for operators of the form \eqref{def:levy}.  Our main theorem is a periodicity result for the bounded solutions of $\Levy^\mu[u]=0$ 
(cf. Theorem \ref{thm:period}).  It states that $\Levy^\mu[u]=0$ if and only if $u$ is periodic with respect to the subgroup $G(\supp (\mu))$ generated by the support. In particular, $\Levy^\mu$ satisfies the Liouville property if and only if $G(\supp (\mu))$ is dense (cf. Theorem \ref{thm:main-d}). 
If $d=1$,  we give an equivalent condition in terms of irrationality of the ratio between points in the support (cf. Theorem \ref{thm:main}). 

\medskip

We also precisely identify  the form of the operators $\Levy^\mu$ which do \textit{not} satisfy the Liouville property. Roughly speaking, this property \textit{fails} if and only if 
 $$
\supp (\mu) \subseteq \R^{n} \times \Z^m,
$$ 
for some $n<d$, $m \leq d-n$, and up to some change of coordinates.  If $d=1$ the corresponding measures $\mu$ are thus series of Dirac masses (cf. Proposition \ref{cor:consequences}), and in the general case $d \geq 1$ we get series of measures of dimension less or equal $d-1$ (cf. Theorem \ref{thm:consequences-d}).


\medskip

In particular,  every   one dimensional  nonpurely discrete measure  corresponds to a nonlocal operator satisfying the Liouville theorem.  This includes  measures whose supports contain intervals like absolutely continuous measures (e.g. 
the fractional Laplacian) or singular continuous Cantor measures.  
 Certain atomic measures may have the Liouville property as well like e.g.  $\mu(z)=\sum_{\pm} (\delta_{\pm a}(z)+\delta_{\pm b}(z))$  if $\frac{b}{a} \notin \Q$. 
To easily verify whether the Liouville theorem holds  if $d=1$,  we provide an algorithmic procedure in relation with irrational numbers (cf. Corollary \ref{item:method}).  
%
%

\medskip

 For general dimension $d \geq 1$,  measures containing balls in their supports like absolutely continuous measures enjoy the Liouville property. Other interesting examples are mean value operators with which we can recover the classical Liouville theorem (cf. Example \ref{ex:surfaceMultiD}), purely atomic measures with few points in the supports satisyfing a so-called Kronecker's approximation condition from group theory (cf. Corollary \ref{coro:CharKron}), or certain nonstandard discretizations of the Laplacian on nonuniform grids (cf. Example \ref{ex:discLapMultiDbis}). 

\medskip

Our  proofs are based on simple reasoning on irrational numbers, certain results on groups (cf. e.g. \cite{EfHaCL80,God87,Mar03,GoMo16} and the references therein), and pure PDEs arguments of  propagation of maximum (cf.
\cite{Cov08,Cio12,ShHa16,DTEnJa17b,DTEnJa17a,HuDuWu18, BrCoHaVa17,BrCo18}). 

\medskip

Note finally that our periodicity result Theorem \ref{thm:period} not only implies the Liouville theorem for sufficiently nice $\Levy^\mu$, but also charaterizes the bounded solutions of $\Levy^\mu[u]=0$ in every  case.  It is thus a natural extension of the classical Liouville theorem for general nonlocal operators. Moreover, it is possible to obtain a similar result for  general generators  of L\'evy processes, or equivalently general linear operators with constant coefficients satisfying the positive maximum principle a la  Courr\`ege.  This wider class includes local and nonlocal operators possibly nonself-adjoint. In a work in progress \cite{AlDTEnJa18}, we are writing a complete characterization of such operators satisfying the Liouville theorem as a byproduct of an extended version of the periodicity result Theorem \ref{thm:period}.

%
%



\medskip

\bigskip

 \noindent\textbf{Outline of the paper.}  General assumptions are presented in Section \ref{sec:as}.  Section \ref{sec:mainresults} concerns 
the  case where $d=1$ ($1$--$d$ for short)  whose main results are Theorem \ref{thm:main} and Corollary \ref{item:method} followed by representative examples. These results are obtained independently of the  case $d \geq 1$ (multi--$d$ for short). The analysis in the latter case  starts in Section \ref{sec:group} with the general periodicity result Theorem \ref{thm:period}. The classifications for the Liouville theorem are given in Theorems \ref{thm:main-d} and \ref{thm:consequences-d} (cf.  also Proposition \ref{cor:consequences}  for further  results  in $1$--$d$). Section \ref{sec:ex} gives representative multi--$d$ examples, and technical or complementary results are postponed in appendices. 

 
\subsubsection*{General notation}


The symbol $\wedge$ denotes $\min$.  The symbol $\subseteq$ denotes inclusion of sets and $\not\subseteq$ is the negation.   We define $\limsu:= \limsup$. 

\medskip

 Let $B_r(x)$  denote the ball of center $x$   and  radius $r$.  A point $x \in S \subseteq \R$ is \emph{isolated} in  the set $S$  if there exists $r >0$ such that $B_r(x) \cap S=\{x\}$. If every point of $S$ is isolated, $S$ is \emph{discrete}. 

\medskip

The function spaces $C_\textup{b}^\infty (\R^d) $ and $C_\textup{c}^\infty (\R^d) $ are the respective spaces of  smooth functions with bounded derivatives of all orders  and of smooth functions with compact support{s}.

%
%

\subsubsection*{{Measures on $ \R^d  \setminus \{0\}$}}

We only consider nonnegative  Radon  measures  $\mu$ on $\Omega:= \R^d  \setminus \{0\}$,  i.e. Borel measures finite on compact sets.  Their {\it supports} are defined as 
\begin{equation*}
{\supp(\mu)} = \big\{z\in  \Omega  \ : \ \mu(B_r(z))>0, \ \forall r>0\big\}.
\end{equation*}
By definition, ${\supp (\mu)}$ is a closed set  of $\Omega= \R^d  \setminus \{0\}$ and it does not contain zero.  To avoid confusion, its closure in $\R^d$ (which may contain zero) is denoted by   $\overline{\supp (\mu)}^{\R^d}$.  
For any   $a \in \R ^d$,  we denote by $\delta_{a}(z)$ the Dirac measure at $z=a$.  A measure ${\mu_{d}}$ is {\em discrete}  (or \textit{atomic})  if there exists a countable set  $D \subseteq \Omega$  such that 
\begin{equation}\label{def-discrete-meas}
{\mu_{d}}\big(\Omega \setminus D\big)=0.
\end{equation}
 The  support of a discrete measure does not need to be discrete. Here is a standard result that will be needed  (cf. e.g. \cite{Rud87}). 

\begin{lemma}\label{lem:discrete-meas}
If \eqref{def-discrete-meas} holds then 
$
{\mu_d}(z)=\sum_{a \in D} \omega(a) \delta_{a}(z)$ with $\omega(a)=\mu(\{a\})$.
\end{lemma}

 \section{Standard assumptions and examples}\label{sec:as}


We assume that  $\Levy^{\mu}$ is an operator of the form \eqref{def:levy} for some Radon measure $\mu$ on $ \R^d  \setminus \{0\}$ satisfying
\begin{equation}\label{as:mus}\tag{$\textup{A}_{\mu}$}
{ \mbox{ $\mu\geq0,$}\quad\mbox{$\mu$ is symmetric,} \quad \mbox{and}  \quad \int (|z|^2 \wedge 1) \dd \mu(z)<+\infty. }
\end{equation}
 
\begin{remark}
The  symmetry means that $\mu(B)=\mu(-B)$ for any Borel set. In particular, $\Levy^\mu$ is self-adjoint and  the principal value in \eqref{def:levy} makes sense for $u \in C^2_\textup{b}( \R^d )$ via the formula
\begin{equation*}
\begin{split}
& \Levy^\mu[u](x) =  \lim_{r \rightarrow 0} \int_{|z|>r} \big(u(x+z)-u(x)\big) \dd \mu(z)\\
& = \int_{0<|z| < r_0} \big(u(x+z)-u(x)-z \cdot \nabla u   (x)\big) \dd \mu(z)\\
& \quad +\int_{|z| \geq r_0} \big(u(x+z)-u(x)\big) \dd \mu(z)
\end{split}
\end{equation*}
valid for any $r_0>0$. For general $u \in L^\infty( \R^d )$ we say that $\Levy^\mu[u]=0$ in $\mathcal{D}'( \R^d )$ if 
\begin{equation}\label{eq:DistEqual0}
\int_{ \R^d }u(x)\Levy^\mu[\psi](x)\dd x=0 \quad  \forall \psi\in C_\textup{c}^\infty( \R^d ).
\end{equation}
\end{remark}

\begin{remark}
\label{rem:decom}
By the Lebesgue decomposition theorem (cf. e.g. \cite{Rud87}), $\mu$ can be written in a unique way as 
$$
\dd \mu(z)=f(z) \dd z+\dd \mu_s(z)+\dd \mu_d(z),
$$
for some
\begin{enumerate}[{\rm (a)}]
\item Borel measurable and locally integrable (on $\R^d \setminus \{0\}$) function $f \geq 0$,
\item measure $\mu_s \geq 0$ satisyfing 
$$
\begin{cases}
\mu_s(\cdot \cap N)= \mu_s (\cdot) \text{ for some $N$ of null Lebesgue measure,}\\
\mu(\{z\})=0 \text{ for all $z$},
\end{cases}
$$
\item and discrete or atomic  measure $\mu_d \geq 0$. 
\end{enumerate}
The measure $f(z) \dd z$ is the {\em absolutely continuous} part, and $\mu_s$ the   {\em nonatomic}  or {\em diffuse}  singular  part. All the three parts satisfy \eqref{as:mus}. This means that  
$$
f(z)=f(-z) \quad \mbox{and} \quad \int (|z|^2 \wedge 1) f(z) \dd z<+\infty,
$$
and 
$$
\mu_d(z)=\sum_{a \in D} \omega({a}) \big(\delta_a(z)+\delta_{-a}(z)\big)
$$
for some countable set $D \subseteq \R^d \setminus \{0\}$ and   $\omega:D \to \R$  such that
$$
{ \omega>0\quad\mbox{and}}\quad\sum_{a \in D} (|a|^2 \wedge 1) \omega(a) <+\infty.
$$
To identify the discrete measure, we have used Lemma \ref{lem:discrete-meas} with a rewritting of the formula taking into account the symmetry.
\end{remark}

\begin{example}
Here are some classical nonlocal L\'evy operators:
\begin{enumerate}[{\rm (a)}]
\item The fractional Laplacian $-(-\Delta)^\frac{\alpha}{2}$ with $\alpha \in (0,2)$ corresponding to an absolutely continuous L\'evy measure  (cf. e.g. \cite{Sat99}) of the form 
$$ 
\dd\mu(z)=\frac{c_{d,\alpha}}{|z|^{d+\alpha}} \dd z;
$$
\item the relativistic Schr\"odinger operator $m^\alpha I-(m^2I-\Delta)^\frac{\alpha}{2}$ with $m>0$ and $\alpha\in(0,2)$ corresponding to an absolutely continuous L\'evy measure of the form
$$
{\dd \mu(z)=c_{d,\alpha,m}\frac{K_{(d+\alpha)/2}(m|z|)}{|z|^{(d+\alpha)/2}}\dd z,} 
$$
where $K$ is a modified Bessel function (cf. \cite[Remark 7.3]{FaFe15}); 
\item the convolution operator $u \mapsto J \ast u-u$ with   $J \geq 0$, $J(z)=J(-z)$, and $\int J(z) \dd z=1$, corresponding to the absolutely continuous L\'evy measure
$
\dd \mu(z)=J(z) \dd z
$   (cf. e.g. \cite{Cov08,A-VMaRoT-M10}) ;
\item the $1$--$d$ discrete Laplacian 
$$
  \Delta_h u(x) := \frac{u(x+h)+u(x-h)-2u(x)}{h^2}  
$$
corresponding to the atomic L\'evy measure
$
\mu(z)=\frac{1}{  h^2 } \left(\delta_h(z)+\delta_{-h}(z)\right);
$
\item
the general multi--$d$ discrete and self-adjoint diffusion operator   (cf. \cite{DTEnJa18a,DTEnJa18b})
$$
\Levy_h [u](x):= \sum_{a \in D} \omega(a) \big(u(x+a)-u(x-a)-2 u(x) \big)
$$ 
 for some $D$ and $\omega(\cdot)$ as in Remark \ref{rem:decom}, corresponding to the atomic L\'evy measure
$
\mu(z)=\sum_{a \in D}   \omega(a)  (\delta_a(z)+\delta_{-a}(z));
$
\item or the mean value operator   associated to the Laplacian 
$$
\mathcal{M}[u](x):=\frac{1}{c_d} \int_{\partial B_1(0)} \big(u(x+z)-u(x)\big) \dd \sigma(z), 
$$
where $\dd \sigma(z)$ is the surface
measure of the unit sphere and $c_d$ its area. 
\end{enumerate}
\end{example}

%
%

 \section{Characterization in $\R$ in terms of irrational numbers}

 This section is devoted to the dimension $d=1$. The main results Theorem \ref{thm:main} and Corollary \ref{item:method} are given in Section \ref{sec:mr1d} and their proofs in Sections \ref{sec:NecCond} and \ref{sec:SufCond}.


\label{sec:mainresults}

\subsection{Reminders on integers and irrational numbers}

For any $n,m \in \Z \setminus\{0\}$, we let  $\gcd(n,m)$ denote the greatest common  positive divisor between $n$ and $m$.  The integers $n$ and $m$ are {\it coprimes} if $\gcd(n,m)=1$.  Let us also precise that we use the notation $\N_+:=\N \setminus \{0\}$.  

\medskip

Let us recall the B\'ezout identity that will be needed later (cf. e.g. \cite{God87}).

\begin{lemma}\label{Bezout}
Two integers $n,m \in \Z \setminus \{0\}$ are coprimes if and only if $n \Z+m\Z=\Z$.
\end{lemma}

Here is another basic result that will be needed (cf. e.g. \cite{God87}).

\begin{lemma}\label{lem:dense}
For any $a,b \in \R \setminus \{0\}$, $\overline{a \Z+b \Z}=\R$ if and only if $\frac{b}{a} \not\in \Q$. 
\end{lemma}

Lemma \ref{lem:dense} will be fundamental for our analysis in dimension $d=1$, and for completeness we recall its proof in Appendix \ref{proof:density}. 

\subsection{
An explicit condition on the L\'evy measure}

Along with $\eqref{as:mus}$, the following condition will turn out to be  necessary and sufficient 
for $\Levy^\mu$   in $\R$  to satisfy the Liouville theorem:
\begin{equation}\label{as:LiCo}\tag{A$_{\textup{L}}$}
\text{$\exists a\in  \supp (\mu) $}\quad\text{such that} \quad \supsupp \q(a,b)=+\infty,
\end{equation}
where $\q:  \R^2 \setminus \{(0,0)\} \to \N_+\cup\{+\infty\}$ is defined as
\begin{equation}\label{def:q}
\q(a,b):=
\begin{cases}
 +\infty &  \textup{if $\frac{b}{a}\not\in \Q$,}\\[1ex]
q  &   \textup{if $\frac{b}{a}= \frac{p}{q}\in \Q$ with  $|p|, q \in \N_+$  and $\gcd(p,q)=1$.}
\end{cases}
\end{equation}

Note that: 

\begin{lemma}\label{lem:rw}
Let $\mu$ be a nonnegative Radon measure on $\R \setminus \{0\}$ such that \eqref{as:LiCo} holds. Then $
\supsupp \q(a,b)=+\infty
$ for all $a \in \supp (\mu)$.
\end{lemma}

\begin{proof}
Assume $\supsupp \q(a_0,b)=+\infty$ for some $a_0$ and take any other fixed $a \in \supp (\mu)$. If $\frac{b}{a} \notin \Q$ for some $b \in \supp (\mu)$, then $\q(a,b)=+\infty$ and the proof is complete. In the other case, 
$$
\frac{b}{a}=\frac{b}{a_0} \frac{a_0}{a}=\frac{p_{b}}{\q(a_0,b)} \frac{\q(a_0,a)}{p_a}
$$
for some integers $p_a,p_b \neq 0$ such that 
$$
\gcd(p_a,\q(a_0,a))=1=\gcd(p_b,\q(a_0,b)). 
$$
Hence
$$
\q(a,b) =  \underbrace{\left(\frac{\q(a_0,b)}{\gcd(\q(a_0,a),\q(a_0,b))}\right)}_{\geq \frac{\q(a_0,b)}{\q(a_0,a)}} \, \, \underbrace{\left(\frac{|p_a|}{\gcd(|p_a|,|p_b|)}\right)}_{\geq 1},
$$
and therefore  since $1 \leq \q(a_0,a)<+\infty$, 
$$
{{\supsupp \q(a,b) \geq \frac{\supsupp \q(a_0,b)}{\q(a_0,a)} =+\infty.}}
$$
The proof is complete.
\end{proof}

It  will 
be    useful to refer to 
 the negated version of \eqref{as:LiCo}:
\begin{equation}\label{as:NoLiCo}\tag{NA$_{\textup{L}}$}
\text{$\forall a\in  \supp (\mu) $,} \quad \supsupp \q(a,b)<+\infty.
\end{equation}

\begin{remark}
\label{rem:rw}
By Lemma \ref{lem:rw}, it suffices that $\supsupp \q(a,b)<+\infty$ for just one $a \in \supp (\mu)$ to deduce that \eqref{as:NoLiCo} is satisfied.
\end{remark}

\begin{remark}\label{rem:SatisfyLiCo}
As concerning \eqref{as:LiCo}, there are two ways of satisfying this assumption:
\begin{enumerate}[\rm (i)]
\item \label{rem:SatisfyLiCo(i)} Either the ratio $\frac{b}{a}$ is irrational for some $a,b\in{\supp (\mu)}$, or
\item \label{rem:SatisfyLiCo(ii)} there are $a\in{\supp (\mu)}$ and $\{b_{ n }\}_{{ n }} \subseteq { \supp (\mu)}$ such that $q_{ n } \to+\infty$ as ${ n }\to +\infty$ where $\frac{b_{ n }}{a}=\frac{p_{ n }}{q_{ n }}$ for ${ |p_n| },q_{ n }\in \N_+$ with $\gcd ({ p_n },q_{ n })=1$.
\end{enumerate}
\end{remark}

\begin{remark}
Later, we will see that \eqref{as:LiCo} (resp. \eqref{as:NoLiCo}) is a necessary and sufficient condition on $\supp (\mu)$ to generate a dense (resp. discrete) subgroup of $\R$ (cf. Remark \ref{rem:c1}\eqref{item:c1}--\eqref{item:c2}). 
\end{remark}

\subsection{Main results and examples}\label{sec:mr1d}

Our first main result in $\R$  is the following:
\begin{theorem}\label{thm:main}
Assume  $d=1$,  \eqref{as:mus}, and let $\Levy^\mu$ be defined by \eqref{def:levy}. The following statements are equivalent:
\begin{enumerate}[{\rm (a)}]
\item \label{label:main1} If $u\in L^\infty(\R)$ satisfies $\Levy^\mu[u]=0$ in $\mathcal{D}'(\R)$, then $u$ is constant.
\item \label{label:main2}$\mu$ satisfies the property \eqref{as:LiCo}.
\end{enumerate}
\end{theorem}

The parts \eqref{label:main1} $\Rightarrow$ \eqref{label:main2} and \eqref{label:main2} $\Rightarrow$ \eqref{label:main1} are proved separately in Sections \ref{sec:NecCond} and \ref{sec:SufCond}. 
In practice, we   can check algorithmically  if a $1$--$d$ operator enjoys the Liouville property with the help of the following procedure:

\begin{corollary}[A practical method of verification in $1$--$d$]
\label{item:method} 
Assume \eqref{as:mus} and  $d =1$.  Then:
\begin{enumerate}[{\rm (a)}]
\item\label{rem:method} If $\overline{\supp (\mu)}^\R$ contains an interval or just one accumulation point, then $\Levy^\mu$ satisfies the Liouville theorem (in the sense of Theorem \ref{thm:main}\eqref{label:main1}).
\item\label{arxiv-detail} If  part \eqref{rem:method} does not hold, then:
\begin{enumerate}[{\rm (i)}]
\item\label{item:method-finite} If $\supp (\mu)$ contains only a finite number of points,
$$
\supp (\mu)=\{-a_N,\dots, -a_1,a_1,\dots,a_N\} \quad \mbox{with} \quad 0<a_1<\dots<a_N,
$$
then we have the Liouville property if and only if   $\frac{a_n}{a_1}\not\in \Q$ for some $a_n$;    
\item\label{item:method-infinite} if $\supp (\mu)$ has an infinite number of points, 
$$
\supp (\mu)=\{-a_n,a_n\}_{n \geq 1} \quad \mbox{with} \quad 0<a_1<a_2<\dots,
$$
then we write $\frac{a_n}{a_1}=\frac{p_n}{q_n}$ in irreducible form 
and we have the Liouville property if and only if $\sup_{n \geq 1} q_n=+\infty$.
\end{enumerate}
\end{enumerate}
\end{corollary}
 
\begin{proof}
 Part    \eqref{arxiv-detail} is an immediate consequence of Theorem \ref{thm:main}  and  Remarks \ref{rem:rw} and \ref{rem:SatisfyLiCo}. 
For \eqref{rem:method}, we need Corollary \ref{cor:atomicMes} that will be proved later. Corollary \ref{cor:atomicMes} implies that if \eqref{as:NoLiCo} holds then $\supp (\mu) \subseteq g \Z$ for some $g > 0$. But such an inclusion is not possible if $\overline{\supp (\mu)}^\R$ has an accumulation point. This means that necessarily \eqref{as:LiCo} holds and we conclude again by Theorem \ref{thm:main}.
\end{proof}



%
%


\begin{example}[$1$--$d$ continuous L\'evy measures satisfy Liouville]
Any measure with either an absolutely continuous part or a nontrivial  diffusive  singular part (cf. Remark \ref{rem:decom})  satisfies   the Liouville theorem by Corollary \ref{item:method}\eqref{rem:method}. 
Examples are
%
the fractional Laplace measure,
%
%
relativistic Schr\"odinger operators, convolution operators, or even purely Cantor measures.   
\end{example}

\begin{example}[Discrete measures satisfying Liouville]
\begin{enumerate}[{\rm (a)}]
\item  A nice example of a discrete operator satisfying the Liouville theorem is the following nonstandard discretization of the Laplacian 
\begin{equation}\label{ex:ns}
\begin{split}
{{\Levy_h[u]}} & = 
\frac{u(x+h)+u(x-h)-2 u(x)}{2h^2}\\
& \quad +\frac{u(x+\pi h)+u(x-\pi h)-2 u(x)}{2\pi^2 h^2}.
\end{split}
\end{equation}
This follows from \eqref{item:method-finite} of   Corollary \ref{item:method}  because  $\frac{\pi h}{h}\not\in\Q$.  We refer the reader to  Section \ref{sec:exfinitediffMultiD}  for examples in multi-$d$.
\item  Other typical examples  are nonzero order discrete operators  with unbounded measures at $0$  ($\mu(B_1(0))=+\infty$ but $\int (|z|^2 \wedge 1) \dd \mu(z)<+\infty$)  like 
\[
  \mu=\sum_{ n \geq 1 }  \big(\delta_{\frac{1}{n}}+\delta_{-\frac{1}{n}} \big) \quad \Big(\text{here $\int (|z|^2 \wedge 1) \dd \mu(z)=\sum_{n \geq 1} \frac{{{2}}}{n^2}<+\infty$} \Big).
\] 
  The Liouville property  follows from   \eqref{rem:method} of Corollary \ref{item:method}   ($0$ is an accumulation point of $\overline{\supp (\mu)}^\R$).   See Section \ref{sec:exunboundedMultiD} for a similar discussion in multi-$d$.
\item 
A last illustrative example is given by   a discrete measure with infinitely many   rational and   isolated points:   
\begin{equation*}
\mu =\sum_{ n \geq 1 }  \omega_n \big(\delta_{a_n}+\delta_{-a_n}\big)
\end{equation*}
 for  $\{a_n\}_{n \geq 1}$ and $\{\omega_n\}_{n \geq 1}$ satisfying 
\begin{equation*}
\omega_n > 0, \quad \sum_{ n \geq 1 } \omega_n<+\infty, \quad \mbox{and} \quad a_n=\frac{n^2+1}{n}.
\end{equation*} 
Now the Liouville property follows from \eqref{item:method-infinite} of   Corollary \ref{item:method}.   Indeed $\gcd(n^2+1,n)=1$ by
the B\'ezout identity  Lemma \ref{Bezout} 
  and therefore since $\frac{a_n}{a_1}=\frac{n^2+1}{2n}$, we have
$$
\q (a_1=1,a_n) \geq n \to +\infty \quad \mbox{as} \quad n \to +\infty.
$$
 
\end{enumerate}
\end{example}

\begin{example}[Counterexamples for standard finite difference operators]
 By   Corollary \ref{item:method}\eqref{item:method-finite},  the usual discretization of the $1$--$d$ Laplacian  
$$
\Delta_h u(x)=\frac{1}{h^2}\big(u(x+h)+u(x-h)-2u(x)\big)
$$ 
does not have the Liouville theorem, neither  will any other discretization of the form 
$$
{  \Levy_h[u](x)= \sum_{n \geq 1} \omega_{n} \big(u(x+n h)+u(x-n h)-2u(x)\big). }
$$
To preserve the Liouville property for a numerical scheme,   one could  consider nonuniform grids with irrational ratios between nodes as in \eqref{ex:ns}.
\end{example}



\subsection{Proof of the necessary condition}
\label{sec:NecCond}

 Let us now prove that \eqref{label:main1} $\Rightarrow$ \eqref{label:main2}  in Theorem \ref{thm:main}. We  will proceed by the contrapositive statement. This means that if \eqref{as:NoLiCo} holds, we can find a nonconstant function $U\in L^\infty(\R)$ such that $\Levy^\mu[U]=0$ in $\mathcal{D}'(\R)$.  Let us start with elementary results. 


The first lemma shows that  $0 \notin \overline{\supp (\mu)}^\R$. 

\begin{lemma}\label{lem:NoAc0}
Assume  \eqref{as:mus}, $d=1$, and  \eqref{as:NoLiCo}.  Then  either $\mu=0$ (trivial measure) or 
\[
a_{1}:=\inf\left\{b>0 \ : \ b\in { \supp (\mu) } \right\} >0.
\]
\end{lemma}

\begin{proof}
Assume by contradiction that  $\mu$ is not the trivial measure  and $a_1=0$. Then, by definition there exists a sequence $\{b_n\}_{ n}\subseteq { \supp (\mu) }\cap (0,+\infty)$ such that $b_n \rightarrow 0$ as $n\to+ \infty$. Now take any $a\in { \supp (\mu) }\cap (0,+\infty)$. If $\frac{b_{n_0}}{a}\not\in \Q$ for some  $n_0 \in \N_+$  then $\q(a,b_{n_0})=+\infty$ which contradicts  \eqref{as:NoLiCo}.  Otherwise, let $q_n\in \N_+$ be such that $\frac{b_n}{a}=\frac{p_n}{q_n}$ with $p_n \in \N_+$ and $\gcd(p_n,q_n)=1$, and therefore
\[
\q(a,b_n)=q_n.
\]
Then, since $p_n\geq1$ and $b_n \rightarrow 0$, we have
\[
 \supsupp  \q(a,b)\geq \q(a,b_n)=q_n= \frac{a\, p_n }{b_n}\geq \frac{a}{b_n}\stackrel{n \to \infty}{\longrightarrow}+\infty,
\]
and again we reach a contradiction with  \eqref{as:NoLiCo}.  
%
%
\end{proof}

\begin{corollary}
\label{cor:atomicMes}  
  Assume \eqref{as:mus}, $d=1$, and \eqref{as:NoLiCo}. Then  there is $g > 0$ such that 
$\supp (\mu) \subseteq g \Z$. 
\end{corollary}

\begin{proof}
 If $\mu=0$ then $\supp(\mu)=\emptyset$ and the result is trivial.  In the other case,   let $a_1>0$ be given by Lemma \ref{lem:NoAc0}. By \eqref{as:NoLiCo} 
$$
\supsupp \q (a_1, b)<+\infty,
$$
where  $\q(a_1,b) \in \N_+$ for any $b \in \supp (\mu)$ by the definition \eqref{def:q}. This supremum is therefore a maximum and there exists $\overline{q} \in \N_+$ such that 
\[
\overline{q}=\max_{b \in \supp (\mu)} \q (a_1, b).
\]
For any $b \in \supp (\mu)$, $\frac{b}{a_1} \in \Q$ by \eqref{as:NoLiCo} and there exists an integer $p_b$ such that 
$$
b=a_1 \frac{b}{a_1}=a_1 \frac{p_b}{\q(a_1,b)}=\frac{a_1}{\overline{q}!} \frac{\overline{q}! \, p_b}{\q(a_1,b)}.
$$ 
Note that $\q(a_1,b)$ divides $\overline{q} !$ by the above definition of $\overline{q}$. Defining $g:=\frac{a_1}{\overline{q}!}>0$ we conclude that $b \in g \Z$.  Since $b$ was arbitrary  in $\supp (\mu)$ and the choice of $g$ did not depend on $b$, the proof is complete.
\end{proof}


The following proposition completes the proof of \eqref{label:main1} $\Rightarrow$ \eqref{label:main2} in Theorem \ref{thm:main}. 

\begin{proposition}
\label{prop:counterExample}
Assume \eqref{as:mus},  $d=1$,  and \eqref{as:NoLiCo}. Then there exists a nonconstant function $U\in C^\infty_\textup{b}(\R)$ such that  $\Levy^\mu[U]= 0$ in $\R$. 
\end{proposition}

\begin{proof}
By Corollary \ref{cor:atomicMes},  there exists $g>0$ such that $\supp (\mu) \subseteq g \Z$ and   thus  $\mu$ is a discrete measure. By Lemma \ref{lem:discrete-meas} and Remark \ref{rem:decom},  
$$
\Levy^\mu[U](x)=\sum_{n \geq 1} \omega_n \big(U(x+n g)+U(x-ng)-2 U(x)\big),
$$
with $\omega_n \geq 0$ satisfying $\sum_{n \geq 1} \omega_n<+\infty$. Hence, every $g$-periodic $U$ satisfy $\Levy^\mu[U]= 0$ like e.g. $U(x)=\cos (\frac{2 \pi}{g}x)$. 

\end{proof}
%
%

\subsection{Proof of the sufficient condition}
\label{sec:SufCond}


 We will now  show the Liouville theorem given by the implication \eqref{label:main2} $\Rightarrow$ \eqref{label:main1} in Theorem \ref{thm:main}. Roughly speaking, we will use techniques from \cite{Cov08,Cio12,BrCoHaVa17} (cf. Theorem \ref{prop:Ciomaga} and Lemma \ref{lem:cov}) to deduce that $u$ is constant on some abstract set, that we will rigorously identify to be the whole domain $\R$ under \eqref{as:LiCo}. Later, we will ameliorate the techniques of \cite{Cov08,Cio12,BrCoHaVa17} to get a more general periodicity result valid for any dimension $d \geq 1$ (cf. Theorem \ref{thm:period}). This general result will imply the multi--$d$ Liouville theorem as a byproduct (cf. Theorem \ref{thm:main-d}). But let us first use more standard arguments and proceed in a serie of representative examples to shed light on the main ideas. 




\subsubsection*{Example 1: Only two points 
with irrational ratio}

Let us  consider a prototypical operator satisfying \eqref{as:LiCo}  as in Remark  \ref{rem:SatisfyLiCo}\eqref{rem:SatisfyLiCo(i)}.  
 Let 
\begin{equation}\label{eq:ex1}
\mu=    \omega_1  \big(\delta_a+\delta_{-a}\big)+   \omega_2  \big(\delta_b+\delta_{-b}\big) \quad \mbox{for} \quad a,b,\omega_1,\omega_2>0 \quad \mbox{with} \quad \frac{b}{a} \not\in\Q.
\end{equation}
Here ${\supp (\mu)}=\{-b,-a,a,b\}$ since $a \neq b$. 

\begin{proposition}\label{thm:ex1}
Assume $\mu$ is given by \eqref{eq:ex1} and  $u\in C_\textup{b}^{{ \infty}}(\R)$  solves  $\Levy^\mu[u]= 0$ in $\R$. 
If $u$ has a global maximum then $u$ is constant.
\end{proposition}

%
%

 Here the proof easily follows from the reasoning in \cite{Cov08,Cio12} and Lemma \ref{lem:dense}. Let us give detail for completeness. 

\begin{proof}[Proof of Proposition \ref{thm:ex1}]
Without loss of generality assume that 
$$
M:=u(0)= { \max  u. } 
$$ 
  First note that 
\begin{equation*}
\begin{split}
0&=\Levy^\mu[u](0)=    \omega_1  \sum_{\pm}(u(\pm a)- M)+   \omega_2  \sum_{\pm}(u(\pm b)- M),
\end{split}
\end{equation*}
and since each term in the sum is nonpositive by definition of $w_1,w_2$ and $M$, we have that $u(\pm a)=u(\pm b)=M$. Then if $u(mb+ka)=M$ for some $m,k\in \Z$, again
\begin{equation*}
\begin{split}
0&=\Levy^\mu[u](mb+ka)\\
&=    \omega_1  \sum_{\pm}(u(mb+(k\pm1)a)- M)+   \omega_2  \sum_{\pm}(u((m\pm1)b+ka)- M),
\end{split}
\end{equation*}
  and again by the nonpositivity of the terms we get
\[
u(mb+(k\pm1)a)
=u((m\pm1)b+ka)
=M.
\]
By induction, we find that
\[
u( x )=M={ \max u } \quad  \forall x \in a \Z + b\Z. 
\]
Since $u$ is continuous  the proof is complete by Lemma \ref{lem:dense}.
%
%
\end{proof}

\subsubsection*{Example 2: Infinitely many points 
with asymptotic irrational ratio}

Let us  consider a prototypical operator satisfying \eqref{as:LiCo} as in Remark \ref{rem:SatisfyLiCo}\eqref{rem:SatisfyLiCo(ii)}.  Let 
\begin{equation}\label{eq:ex2}
\mu(z)=\sum_{n \geq 1} \omega_n \big(\delta_{a_n}(z)+\delta_{-a_n}(z) \big) \quad \mbox{with} \quad \omega_n > 0, \quad \sum_{n \geq 1} (a_n^2 \wedge 1) \omega_n <+\infty,
\end{equation}
as well as 
\begin{equation}\label{eq:ex2-as}
a_n>0 \quad \mbox{and} \quad \sup_{n \geq 1} \q (a_{1},a_{n})=+\infty. 
\end{equation}



\begin{proposition}\label{thm:ex2}
Assume \eqref{eq:ex2}--\eqref{eq:ex2-as} and $u\in C_\textup{b}^{ {\infty}}(\R)$ solves $\Levy^\mu[u]= 0$  in $\R$. If $u$ has a global maximum, then $u$ is constant.
\end{proposition}


 Now we need the density result below to prove Proposition \ref{thm:ex2}. 


\begin{lemma}\label{lem:CondMeasureAllPoints}
Assume \eqref{as:mus}, $d=1$, and \eqref{as:LiCo}. Then  
$$
\overline{\bigcup_{a,b \in \supp (\mu)} (a \Z+b \Z)}=\R. 
$$
\end{lemma}

\begin{proof} 
If $\frac{b}{a} \notin \Q$ then the result follows from Lemma \ref{lem:dense}. Assume now that any such ratio is rational. 
 By the B\'ezout identity Lemma \ref{Bezout} and  the definition of $\q(\cdot,\cdot)$ in \eqref{def:q}, we infer that for any $a,b \in \supp (\mu)$ there exists $p \in \Z$  such that   $\frac{b}{a}=\frac{p}{\q (a,b)}$ with $\gcd(p,\q(a,b))=1$, and
\begin{equation*}
\begin{split}
a \Z+b \Z & =a\left(\Z+\frac{b}{a} \Z \right)=a \left(\Z+\frac{p}{\q(a,b)} \Z \right)\\
& =\frac{a}{\q (a,b)} \left(\q(a,b) \Z+p \Z \right)=\frac{a}{\q (a,b)} \Z.
\end{split}
\end{equation*} 
From that identity and \eqref{as:LiCo} the rest of the proof is obvious.
%
\end{proof}

\begin{proof}[Proof of Proposition \ref{thm:ex2}]
Without loss of generality assume that  
$$
M:=u(0)={ \max u. }
$$
 Let us again propagate this maximum as in \cite{Cov08,Cio12}. 
First note that 
\[
0=\Levy^\mu[u](0)=\sum_{n\geq1} \omega_n \big(\underbrace{u(\pm a_n)-M}_{\leq 0}\big).
\]
so that every term is $0$, i.e $u(\pm a_n)=M$ for all $n \geq1 $. Now if $u(m a_{n_1}+ka_{n_0})=M$ for some $m,k\in \Z$, $n_0\geq1$, and $n_1 \geq 1$,  then 
\begin{equation*}
 0=\Levy^\mu[u](m a_{n_1}+ka_{n_0}) 
 =\sum_{n \geq 1} \sum_{\pm} \omega_n \big(\underbrace{u(m a_{n_1}+ka_{n_0}\pm a_n)-M}_{\leq 0}\big),
\end{equation*}
and every term in the sum is again $0$. In particular, for $n=n_0,n_1$, we get
\[
u( m a_{n_1}+(k\pm1)a_{n_0} )=
u( (m\pm1)a_{n_1}+ka_{n_0} )=
M.
\]
 By induction, this shows that for all $n_0\geq1$ and $n_1 \geq 1$, we have
\[
u( x )=M={ \max u }  \quad  \forall x \in a_{n_0}\Z+a_{n_1} \Z.
\]
 Since \eqref{eq:ex2}--\eqref{eq:ex2-as} imply \eqref{as:LiCo}, we can then apply Lemma \ref{lem:CondMeasureAllPoints} to conclude the proof. 
\end{proof}

%
%

%


\subsubsection*{Case of general measures satisfying \eqref{as:LiCo}}

We are now ready to prove that \eqref{label:main2} $\Rightarrow$ \eqref{label:main1} in Theorem \ref{thm:main}.  Let us first recall a general result from  \cite{Cio12} (cf. \cite{Cov08}) which will give us a way of computing the set on which we can propagate the maximum.  
 Let us state it in multi--$d$ for later use. 

\begin{theorem}[Thm. 2 in \cite{Cio12}]\label{prop:Ciomaga}
Assume \eqref{as:mus},  $d \geq 1$,  $u \in C_\textup{b}^\infty(\R^d)$  satisfies $\Levy^\mu[u]= 0$ in $\R^d$, and $u$ has a global maximum at $x_0$. Then $u= u(x_0)$ on $x_0+\overline{\cup_{n\geq0}A_n}$ where
$$
A_0=\{0\} \quad\text{and}\quad A_{n+1}=A_n+\supp (\mu)\quad  \forall n\in \N. 
$$
\end{theorem}

\begin{remark}
This result apply for fully nonlinear PDEs in the context of possibly irregular viscosity subsolutions (cf. \cite{Cio12} for more details).  
\end{remark}

 Theorem \ref{prop:Ciomaga} is roughly speaking for the case where $u$ reaches a global maximum as assumed in Propositions \ref{thm:ex1} and \ref{thm:ex2}. For the general case,  we need  a technique developed in Step 1 of the proof of Lemma 7.2 in \cite{BrCoHaVa17}, that we recall under the form of a lemma.   

\begin{lemma}[{cf. \cite{BrCoHaVa17}}]\label{lem:cov}
Assume \eqref{as:mus} and 
$\overline{\cup_{k \geq 0} A_k}={\R}$  where $\cup_{k \geq 0} A_k$ is   defined in  Theorem \ref{prop:Ciomaga}. If   $v \in C_{\textup{b}}^{{ \infty}}({\R})$ satisfies $\Levy^\mu[v] = 0$  in $\R$   and  $\{x_n\}_{n}$ is a sequence such that $\lim_{n \to +\infty} v(x_n)=\sup v$, then
for any $R \geq 0$,
\begin{equation*}
\limsu_{n \to  +\infty} \min_{ x \in \overline{B}_R(x_n) } v(x)= \sup v.
\end{equation*}
\end{lemma}

\begin{remark}
This result works in $\R^d$ but we will not need it.  The $\limsu$ is also in fact a true limit but we will not need it either. 
\end{remark}


\begin{proof}
Define $v_n(x):=v(x+x_n)$   where $v \in C^{{ \infty}}_b({\R})$.  Then by the Arzel\`a-Ascoli theorem,   $v_n \to v_\infty$  locally uniformly on  $\R$  for some $v_\infty$ along a subsequence  (denoted as the whole sequence to simplify).  Taking another subsequence if necessary, we can assume that the derivatives up to second order   locally uniformly  converge too. Hence, we can easily take the limit in the equation $\Levy^\mu[v_n] = 0$ to deduce that $\Levy^\mu[v_\infty] = 0$  in ${\R}$.  Moreover, $v_\infty\leq \sup v=:M$ and $v_\infty(0)=\lim_{n \to  +  \infty} v(x_n)=M$.  By Theorem \ref{prop:Ciomaga}, we conclude that  $v_\infty = M$ on $\overline{\cup_{k \geq 0} A_k}$ and by assumption $\overline{\cup_{k \geq 0} A_k}={\R}$. Hence for each $R \geq 0$,
 
$$
\max_{\overline{B}_R(x_n)}|v-M|=\max_{\overline{B}_R(0)}|v_n-v_\infty|\to0\quad\textup{as}\quad n\to+\infty
$$
by local uniform convergence of $v_n$.  In particular $\min_{\overline{B}_R(x_n)} v \to M$ (along a subsequence). 

\end{proof}

\begin{proof}[Proof of  Theorem \ref{thm:main}] 
The implication \eqref{label:main1} $\Rightarrow$ \eqref{label:main2} was already proved in Section \ref{sec:SufCond}  (cf. Proposition \ref{prop:counterExample}).  Let us consider the reciprocal assertion. We thus assume that \eqref{label:main2} holds, $u$ is merely bounded and solves $\Levy^\mu[u]=0$ (in $\R$), and we would like to prove that $u$ is a.e. equaled to a constant.

\medskip

\noindent\textbf{1.} \textit{Reduction of the proof to smooth functions.} 

\smallskip

Let $\rho_\varepsilon$ be a mollifier such that $\rho_\varepsilon \in C^\infty_\textup{c}(\R)$ and $\rho_\varepsilon \ast u \to u$  in $\mathcal{D}'(\R)$ as $\varepsilon \rightarrow 0$. If all $\rho_\varepsilon \ast u$ are constant, it will follow that $u$ is constant  since  $u_x$ will be zero as limit in $\mathcal{D}'(\R)$ of $(\rho_\varepsilon \ast u)_x=0$.  But $\rho_\varepsilon \ast u \in C^\infty_\textup{b}(\R)$ and we can also see that $\Levy^\mu[\rho_\varepsilon \ast u] = 0$ by taking $\psi(y)=\rho_\varepsilon(x-y)$ in \eqref{eq:DistEqual0} for fixed $x$. Hence, if we can prove  Theorem \ref{thm:main}[\eqref{label:main2} $\Rightarrow$ \eqref{label:main1}]  for smooth functions like $\rho_\varepsilon \ast u$, the result will follow for merely bounded functions like $u$. Let us  thus  continue by assuming without loss of generality that $u\in C_\textup{b}^\infty(\R)$ solves $\Levy^\mu[u] =  0$  in  $\R$.  

\bigskip

 \noindent\textbf{2.} \textit{Uniform lower bounds on $u_x$.} 

\smallskip

In the rest of the proof, we will prove that $u_x= 0$ which implies that $u$ is constant. We will  only  show that 
$$
M:=\sup u_x\leq0,
$$
  since $v=-u$ solves $\Levy^\mu[v] =0$ and the same reasoning will automatically give us that $\inf u_x=-\sup v_x \geq0$.   Let us start from the equation for $u_x$ obtained by differentiating $\partial_{x}  (\Levy^\mu[u])= 0$,
\begin{equation*}
\Levy^\mu[ u_x]= 0  \quad \textup{in} \quad \R.
\end{equation*}
  Let  $\{x_n\}_{n } \subseteq \R$ and $\{\varepsilon_n\}_{n} \subseteq (0,+\infty)$  such that 
\begin{equation*}
u_x(x_n)
=M-\varepsilon_n,
\end{equation*}
where $\varepsilon_n\to0$ as $n \to+ \infty$.  Let now  $\overline{\cup_{k \geq 0} A_k}$  be defined in Theorem \ref{prop:Ciomaga} i.e. 
$$
A_k=\{0\}+\underbrace{\supp(\mu)+\dots +\supp (\mu)}_{k \text{ times}}.
$$
Since $\supp (\mu)=-\supp (\mu)$, it is easy to see that
$$
a \Z+b \Z \subseteq \cup_{k \geq0}A_k \quad \forall a,b \in \supp (\mu).
$$
 When \eqref{as:LiCo} holds,  Lemma \ref{lem:CondMeasureAllPoints} 
then implies that 
$$
\R=\overline{\bigcup_{a,b \in \supp (\mu)} (a \Z+b \Z)} \subseteq \overline{\bigcup_{k \geq   0 } A_k}.
$$
 An application of Lemma \ref{lem:cov} with $v=u_x$ then shows that for all $R\geq0$
\begin{equation*}
\min_{\overline{B}_R(x_{n})} u_x \to M \quad \mbox{as} \quad n \to +\infty
\end{equation*}
(up to a subsequence).

\bigskip

 \noindent\textbf{4.} \textit{Conclusion.}  

\smallskip

If $M>0$, then for any $ 0<\eta<M$  and $R>2 \|u\|_{ \infty }/( M -\eta)$,  we can take a \textit{fixed} $n$ large enough such that 
\[
\min_{ \overline{B}_R(x_{  n  }) }  u_x  \geq M -\eta>0.
\]
 Writting then that $u(x_n+R)=u(x_n)+\int_{x_n}^{x_n+R} u_x(x) \dd x$, 
\begin{equation*}
\begin{split}
u(x_{  n  }+R) &\geq u(x_{  n  }) + R \min_{ \overline{B}_R(x_{  n  }) }  u_x \\
& > u(x_{  n  }) + \underbrace{\frac{2\|u\|_{ \infty }}{( M -\eta)} \min_{ \overline{B}_R(x_{  n  }) }  u_x }_{ \geq \frac{2 \|u\|_\infty}{M-\eta} (M-\eta)=2 \|u\|_\infty } \geq \|u\|_{ \infty }.
\end{split}
\end{equation*}
We obtain the contradiction that $u(x_n) > \|u\|_\infty$ for this \textit{fixed} $n$. This completes the proof that $u_x = 0$, and thus the proof of the theorem.
%
\end{proof}

\section{Characterization in $\R^d$ in terms of groups and lattices}\label{sec:group}

  
 Let us now consider the dimension $d \geq 1$. Our general periodicity result Theorem \ref{thm:period} is given in Section \ref{sec:p}. Our  characterization results for the Liouville theorem are given in Theorems \ref{thm:main-d} and \ref{thm:consequences-d} of  Sections \ref{sec:cl1} and \ref{sec:cl2}, respectively.  Theorem \ref{thm:main-d} is a generalization of Theorem \ref{thm:main} expressed in terms of subgroups of $\R^d$, and Theorem \ref{thm:consequences-d} precises the general form of the operators for which Liouville fails (cf. also Proposition \ref{cor:consequences} for the case $d=1$).  All the proofs are given (almost) just after the statements of the results.

\subsection{Reminders on groups and lattices}

Let us first recall some basic facts  about  groups that will be needed (cf. e.g. \cite{God87}, \cite[Chapter VII]{Bou98} or \cite[Section 1.1]{Mar03}). 

\begin{definition}
$(G,+)$ is a \textit{subgroup} of $(\R^d,+)$ if $\emptyset \neq G \subseteq \R^d$ and 
$$
{\forall \ x,y \in G, \quad  x+y \in G \quad \mbox{and} \quad -x \in G. }
$$
The \textit{subgroup generated} by a set $S \subseteq \R^d$, denoted $G(S)$, is the smallest group containing $S$.
\end{definition}


\begin{definition}
If two subgroups $G,\tilde{G} \subseteq \R^d$ satisfy $G \cap \tilde{G}=\{0\}$, the sum $G+\tilde{G}$ is \textit{direct} and we write $G+\tilde{G}=G \oplus \tilde{G}$.
\end{definition}

\begin{definition}
A subgroup $\Lambda \subseteq \R^d$ is a \textit{lattice} if it is generated by a certain basis $(a_1,\dots,a_d)$ of the vector space $\R^d$ i.e. 
$
\Lambda=
\oplus_{n=1}^d a_n \Z.
$
A \textit{relative lattice} is a lattice of a vector subspace of $\R^d$.
\end{definition}

\begin{remark}\label{rem:lattice}
\begin{enumerate}[{\rm (a)}]
\item\label{item:countable} A typical example of (full) lattice is $\Lambda=\Z^d$. More generally, every (full)  lattice  are isomorphic to $\Z^d$, up to a change of coordinates from the basis $(a_1,\dots,a_d)$ into the canonical one.
\item By convention, $\Lambda=\{0\}$  is a  relative lattice.
\end{enumerate}
\end{remark}

The following result justifies the introduction of the notion of lattices.

\begin{lemma}\label{lem:lattice}
A subgroup $G$ of $\R^d$ is discrete if and only if it is a relative lattice.
\end{lemma}

This lemma is a consequence of the following  more general result (cf. e.g. \cite[Thm. 1.1.2]{Mar03}):

\begin{theorem}\label{decom:opt}
A closed subgroup $G$ of $\R^d$ is of the form 
$
G=V \oplus \Lambda,
$
for some vector space $V \subseteq \R^d$ and relative lattice $\Lambda  \subseteq \R^d$ such that $V \cap \textup{Span}_\R \Lambda=\{0\}$.
\end{theorem}


 In particular: 

\begin{corollary}\label{pro:group-multid}  
A subgroup $G$ of $\R^d$ is dense if and only if $G \not \subseteq H+c \Z$ for all vector spaces $H \subseteq \R^d$ of codimension $1$ and $c \in \R^d$.
\end{corollary}

\begin{remark}\label{rem:dense-discrete}
In dimension $d=1$, $G$ is either dense or discrete i.e. of the form $G=g \Z$ for some $g \geq 0$ ($g=0$ corresponding to $G=\{0\}$).
\end{remark}


%

%
%

\subsection{Periodicity of the solutions of $\Levy^\mu[u]=0$}\label{sec:p}

\begin{definition}
Given ${S   \subseteq \R^d}$, we say that $u \in L^\infty(\R^d)$ is \emph{$S$-periodic} if 
\[
\int_{\R^d} \big(u(x+s)-u(x)\big) \psi(x) \dd x=0 \quad \forall s\in S, \forall \psi \in C^\infty_\textup{c}(\R^d).
\]
\end{definition}

\begin{remark}
If $u$ is continuous, then $u$ is $S$-periodic if and only if it is $s$-periodic for any $s \in S$.
\end{remark}

Here is our general periodicity result for arbitrary L\'evy measures satisfying only the standard assumption \eqref{as:mus}. 

\begin{theorem}\label{thm:period} Assume \eqref{as:mus} and let   $u\in L^\infty(\R^d)$.    Then $\Levy^\mu[u]=0$ in $\mathcal{D}'(\R^d)$ if and only if $u$ is  $\overline{G(\supp(\mu))}$-periodic   almost everywhere.
\end{theorem}

%


%

\begin{proof}[Proof] For notational simplicity, denote $G:=\overline{G(\supp(\mu))}$. 
  \medskip

\textbf{1.} \textit{Case $u$ smooth.}

\smallskip

Let us first show the theorem for $u \in C_\textup{b}^\infty(\R^d)$.   If $u$ is $G$-periodic then for all $x\in \R^d$ we have that $u(x)=u(x+z)$ for all $z\in\supp(\mu)\subseteq G$. Then
\begin{equation*}
\begin{split}
\Levy^\mu[u](x)&=\textup{P.V.}\int_{ \R^d \setminus \{0\} } \big(u(x+z)-u(x)\big) \dd \mu(z)\\
&=\textup{P.V.}\int_{z\in\supp(\mu)} \big(u(x+z)-u(x)\big) \dd \mu(z)=0,
\end{split}
\end{equation*}
which proves the if part. 

Now assume that $\Levy^\mu[u] =0$ in $\R^d$. 
Fix an arbitrary $g\in G$, and define
\[
v(x):=u(x+g)-u(x).
\]
 If we show that $v =  0$  in $\R^d$, then the proof is complete.  Let us first show that $v \leq 0$. Let 
$$
{M:=  \sup v }
$$ 
and take a sequence $\{x_n\}_{n}$ such that 
\[
v(x_n)\stackrel{n\to \infty}{\longrightarrow} M.
\]
 Let us define  $u_n(x):=u(x+x_n)$ and $v_n(x):=v(x+x_n)$  similarly than in the proof of Lemma \ref{lem:cov} but with the new function $v$,  and note that  
$\Levy^\mu[v_n] = 0$  in $\R^d$. Since $v \in C^\infty_\textup{b}(\R^d)$,  Arzel\`a-Ascoli  theorem implies that there exists $v_\infty$ such that $v_n \to v_\infty$ locally uniformly (up to a subsequence), and taking another subsequence if necessary, we can assume that the derivatives up to second order converge and  pass to the limit in the equation $\Levy^\mu[v_n] = 0$ to deduce that $\Levy^\mu[v_\infty] = 0$  in $\R^d$.   Moreover, $v_\infty$ attains its maximum at $x=0$ since $v_\infty\leq M$ and 
\[
v_\infty(0)=\lim_{n\to+ \infty}v_n(0)=\lim_{n\to +\infty}v(x_n)=M.
\]
 Thus by Theorem \ref{prop:Ciomaga} we have $v_\infty= M$ in $\overline{\cup_{k \geq 0} A_k}$, where 
$$
A_k=\{0\}+\underbrace{\supp(\mu)+\dots +\supp (\mu)}_{k \text{ times}}.
$$
Since $\supp (\mu)=-\supp (\mu)$, it is easy to see that $\cup_{k \geq 0} A_k$ is a subgroup containing $A_1=\supp (\mu)$ and therefore
$$
G=\overline{G(\supp (\mu))} \subseteq \overline{\cup_{k \geq 0} A_k}.
$$
It follows that $v_\infty= M$ in $G$.
 Arguing with $\{u_n\}_{n }$ as for $\{v_n\}_{n }$, there is some $u_\infty$ such that $u_n \to u_\infty$ as $n \to + \infty$ locally uniformly (up to a   further subsequence   if necessary to get the convergence of $u_n$ and $v_n$ along the same sequence).   By construction $v_\infty(x)=u_\infty(x+g)-u_\infty(x)$. For any  $m \in \Z$, we take $x=mg$ to get $M=v_\infty(mg)=u_\infty((m+1)g)-u_\infty(mg)$. By iteration
\[
u_\infty((m+1)g)=u_\infty(mg)+M=\ldots = u_\infty(0)+(m+1)M.
\]
 But since $u_\infty$ is bounded, then the only choice is $M=0$ and thus $v\leq M=0$.  A similar arguments shows that $v \geq 0$ and completes the proof   when $u \in C_\textup{b}^\infty(\R^d)$. 

\bigskip

\textbf{2.} \textit{Case $u$ merely bounded.}

\smallskip

We again consider the convolution $\rho_\varepsilon \ast u$ with a standard mollifier. Recall that $\rho_\varepsilon \ast u \to u$ in $L_{\textup{loc}}^1(\R^d)$ as $\varepsilon \rightarrow 0$  and that $\|\rho_\varepsilon  \ast u \|_\infty \leq \|u\|_\infty$. Now if $u$ is $G$-periodic, then every $\rho_\varepsilon \ast u$ is also $G$-periodic:
\[
 (\rho_\varepsilon \ast u)(x+g)= (\rho_\varepsilon \ast u(\cdot+g))(x)=(\rho_\varepsilon \ast u)(x),\]
  and thus $\Levy^\mu[\rho_\varepsilon \ast u]  = 0$  in $\R^d$  by the previous step. Taking any test function  $\psi \in C^\infty_c(\R^d)$, it is easy to see that $\Levy^\mu[\psi] \in C^\infty_{\textup{b}} \cap L^1(\R^d)$, and this is enough to pass to the limit in 
$$
0=\int_{\R^d} \Levy^\mu[\rho_\varepsilon \ast u] \psi \dd x = \int_{\R^d} \Levy^\mu[\psi] (\rho_\varepsilon \ast u) \dd x\to \int_{\R^d} \Levy^\mu[\psi] u \dd x  \quad \mbox{as} \quad  \varepsilon \rightarrow 0.
$$
This proves that $\Levy^\mu[u] = 0$ in $\mathcal{D}'(\R^d)$. 

Conversely, if $\Levy^\mu[u]= 0$   in $\mathcal{D}'(\R^d)$  then every $\rho_\varepsilon \ast u$ satisfy $\Levy^\mu[\rho_\varepsilon \ast u] = 0$ in $\R^d$. By the  previous  step, every $\rho_\varepsilon \ast u$ is $G$-periodic  and at the limit $\varepsilon \to 0$ we obtain that for any $g \in G$ and $\psi \in C^\infty_{\textup{c}}(\R^d)$, 
$$
0=\int_{\R^d} \big(\rho_{\varepsilon} \ast u(x+g)-\rho_{\varepsilon} \ast u(x)\big) \psi(x) \dd x \to \int_{\R^d} \big(u(x+g)-u(x)\big) \psi(x) \dd x.
$$ 
The proof is complete. 
\end{proof}

\begin{remark}
Notice that to show Theorem \ref{thm:period}, we did not use at all any lower bound on balls as in Lemma \ref{lem:cov}! However, Theorem \ref{thm:period} is very general and will have many consequences on the Liouville theorem. 
\end{remark}

\subsection{The Liouville theorem as a byproduct of periodicity}\label{sec:cl1}

We are now ready to give multi--$d$ necessary and sufficient conditions to have the Liouville theorem.


\begin{theorem}\label{thm:main-d}
Assume $d \geq 1$ and \eqref{as:mus}. Then the following statements are equivalent:
\begin{enumerate}[{\rm (a)}]
\item\label{item:d-t-a}    If $u\in L^\infty(\R^d)$ satisfies $\Levy^\mu[u]=0$ in $\mathcal{D}'(\R^d)$, then $u$ is constant.
\item\label{item:d-t-b}   The additive subgroup generated by $\supp (\mu)$ is dense in $\R^d$.
 
\item\label{item:d-t-bp} For all vector space $V \subseteq \R^d$ of dimension  less or equal  $d-1$ and relative lattice $\Lambda \subseteq \R^d$ such that $V \cap \textup{Span}_\R \Lambda=\{0\}$,  $\supp(\mu) \not\subseteq V \oplus \Lambda$. 
\item\label{item:d-t-c} For all vector space $H \subseteq \R^d$ of dimension $d-1$ and vector $c \in \R^d$, 
$\supp(\mu) \not\subseteq H + c \Z$.
\end{enumerate} 
\end{theorem}

\begin{remark}\label{rem:c1}
 
\begin{enumerate}[{\rm (a)}]
\item\label{item:c1}  By Theorem \ref{thm:main}, we note that if $d=1$ then   
$$
\text{Liouville} \Longleftrightarrow \eqref{as:LiCo} \Longleftrightarrow \text{$G(\textup{supp} (\mu))$ is dense.}
$$
\item\label{item:c2}  By Remark \ref{rem:dense-discrete}, we have the following contrapositive version if $d=1$: 
$$
\text{Not Liouville} \Longleftrightarrow \eqref{as:NoLiCo} \Longleftrightarrow \mbox{$\supp (\mu) \subseteq g \Z$ for some $g > 0$.}
$$ 
This includes the trivial measure $\mu=0$ whose support is empty. 
\item In practice, it might be easier to verify either \eqref{item:d-t-bp} or \eqref{item:d-t-c} than \eqref{item:d-t-b}. In Figure \ref{fig:SeveralSupports2bisbis}, we explain the relations between  (the negated version of) these properties. 
\item If $d>1$, we do not have any algorithmic method like Corollary \ref{item:method} to verify \eqref{item:d-t-bp} and \eqref{item:d-t-c}, but we will identify more precisely the class of operators that do not satisfy these properties (cf. Section \ref{sec:cl2}).
\item Representative examples are given in Section \ref{sec:ex}.
\end{enumerate}
\end{remark}

\begin{figure}[h!]
	\centering
\includegraphics[width=0.65\textwidth]{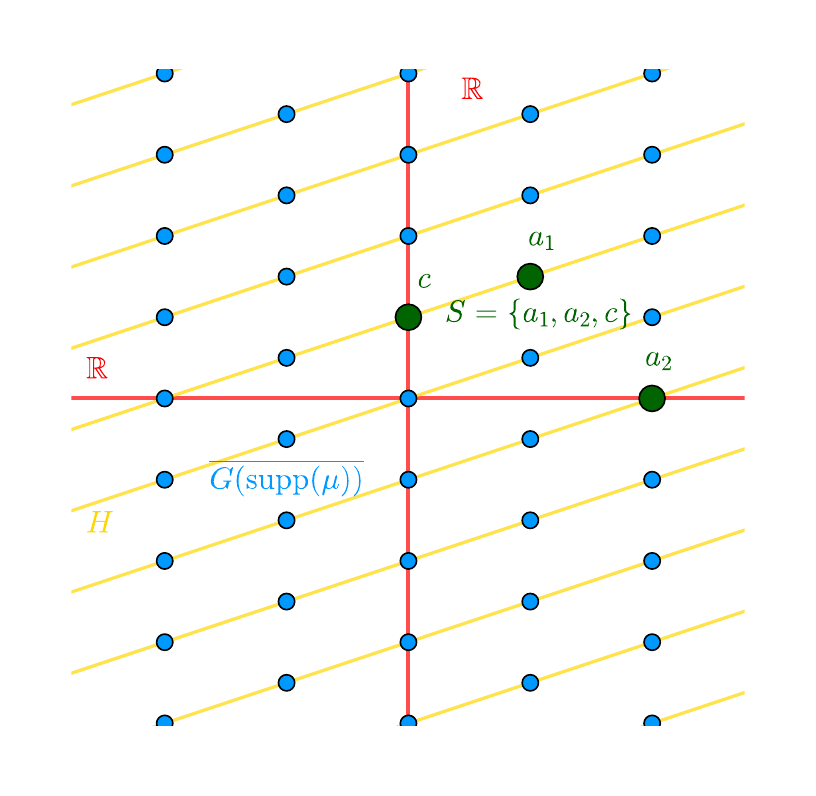}
\caption{  An example where \eqref{item:d-t-bp} and \eqref{item:d-t-c} of Theorem \ref{thm:main-d} fail. We have $\supp (\mu)=S \cup (-S)$ and $\Lambda:=\overline{G(\supp (\mu))}$ is already lattice so $\supp (\mu) \subseteq V \oplus \Lambda$ with $V=\{0\}$ and 
$\supp (\mu) \subseteq H+c \Z$ with $H$ and $c$ as in the picture. }
	\label{fig:SeveralSupports2bisbis}
\end{figure}

\begin{proof}[Proof of Theorem \ref{thm:main-d}]
 By Theorem \ref{decom:opt} and Corollary \ref{pro:group-multid}, \eqref{item:d-t-b} $\Leftrightarrow$ \eqref{item:d-t-bp} $\Leftrightarrow$\eqref{item:d-t-c}.  If now \eqref{item:d-t-b} holds then $u$ is $\R^d$-periodic by Theorem \ref{thm:period} which implies that $u$ is constant. This proves \eqref{item:d-t-b} $\Rightarrow$ \eqref{item:d-t-a}. To conclude, it suffices to prove \eqref{item:d-t-a} $\Rightarrow$ \eqref{item:d-t-c}. Let us do it by a contrapositive reasoning. If \eqref{item:d-t-c} is not true then 
\begin{equation}\label{cep}
\overline{G(\supp(\mu))} \subseteq H+ c \Z,
\end{equation} 
for some vector space $H \subseteq \R^d$ of codimension $1$ and $c \in \R^d$. We can assume that $c \notin H$ because if not we will have \eqref{cep} for any $c \notin H$. In particular,
$$
\R^d=H \oplus \textup{Span}_{\R} \{c\}
$$
and for any $x \in \R^d$ there exists a unique pair $(x_H,\lambda_x) \in H \times \R$ such that
$$
x=x_H+\lambda_x c.
$$ 
Take now 
%
$
U(x)=\cos (2\pi  \lambda_x)
$
and note that for any $h \in H$ and $k \in \Z$,
$$
x+h+ck=\underbrace{(x_H+h)}_{\in H}+\underbrace{(\lambda_x+k)}_{\in \R} c,
$$
so that 
$$
U(x+h+ck) =\cos (2\pi (\lambda_x+k))=\cos (2\pi \lambda_x)= U(x).
$$
This proves that $U$ is $H+c \Z$-periodic and thus $\overline{G(\supp (\mu))}$-periodic. By Theorem \ref{thm:period}, $\Levy^\mu[U]=0$ and the proof is complete.
\end{proof}

\subsection{The class of operators not satisfying Liouville}\label{sec:cl2}


We can have a clearer view of the operators which satisfy the following negated version of the Liouville theorem:
\begin{equation}\label{NL}
\textup{{There  is  a nonconstant $u\in L^\infty(\R^d)$ such that $\Levy^\mu[u]=0$ in $\mathcal{D}'(\R^d)$.}}
\end{equation}

 


To understand the general result in mutli--$d$, it could be useful to first explain what happens if $d=1$. 

\begin{proposition}\label{cor:consequences}
 Assume $d=1$ and \eqref{as:mus}. Then the negated version \eqref{NL} of the Liouville theorem holds if and only if  
\begin{equation}\label{gfnl1d}
\Levy^\mu[u](x) =\sum_{n \geq 1} \omega_n \big(u(x+ng)+u(x-ng)-2 u(x) \big),
\end{equation}
for some $g>0$ and $\omega_n \geq 0$ such that 
$
\sum_{n \geq 1} \omega_n<+\infty,
$
and where \eqref{gfnl1d} can be understood in $\mathcal{D}'(\R)$ or equivalently as a normally convergent serie in $L^\infty(\R)$.
\end{proposition}

\begin{remark}
The largest possible $g > 0$ for which we can write $\Levy^\mu$ in the form \eqref{gfnl1d} is uniquely determined as the generator of $G(\supp (\mu))$. The weights are also entirely determined as $\omega_n=\mu(\{ng\})$ (cf. Lemma \ref{lem:discrete-meas} and Remark \ref{rem:decom}). 
\end{remark}

We omit the detailed verification of Proposition \ref{cor:consequences} because it is in fact an immediate consequence of the arguments already used in the proof of Proposition \ref{prop:counterExample}. Let us focus instead on the general case $d \geq 1$ where we may have series of more general measures than Dirac masses.

\begin{theorem}\label{thm:consequences-d}
Assume $d \geq 1$ and \eqref{as:mus}. Then $\Levy^\mu$  satisfies  \eqref{NL} if and only if it is of the following form:
\begin{equation}\label{gfnl}
\begin{split}
\Levy^\mu[u](x) & = \textup{P.V.} \int_{z \in V \setminus \{0\}} \big(u(x+z)-u(x) \big) \dd \mu_0(z)\\
& \quad + \sum_{a \in \Lambda \setminus \{0\}} \int_{z \in V+a} \big(u(x+z)-u(x) \big) \dd \mu_a(z),
\end{split}
\end{equation}
for some vector space $V \subseteq \R^d$, relative lattice $\Lambda \subseteq \R^d$, and family $\{\mu_a\}_{a \in \Lambda}$ of Radon measures on $\R^d \setminus \{0\}$ such that
\begin{align}
& \label{direct-sum} V \neq \R^d \quad \text{and} \quad V \cap \textup{Span}_{\R} \Lambda=\{0\},\\[1ex]
&\label{levy-d} \mu_a\geq 0 \quad \text{and} \quad \mu_{a}(\cdot)=\mu_{-a}(-\cdot),\\[1ex]
&\label{support-d} \supp (\mu_a) \subseteq V+a, \quad \text{and}\\
&  \label{levy-d-arxiv} \int (|z|^2 \wedge 1) \dd \mu_0(z)+\sum_{a \in \Lambda \setminus \{0\}} \mu_a(\R^d)<+\infty,
\end{align} 
where \eqref{gfnl} has to be understood in $\mathcal{D}'(\R^d)$ and \eqref{levy-d}--\eqref{support-d} hold for all $a \in \Lambda$.
\end{theorem}

\begin{remark}
\begin{enumerate}[{\rm (a)}]
\item A representative example with $V=H$ of codimension $1$ is given in Figure \ref{fig:SeveralSupports2bis}.
\item In \eqref{gfnl}, the L\'evy measure of $\Levy^\mu$ satisfies $\mu=\sum_{a \in \Lambda} \mu_a$ where $\Lambda$ is a relative lattice thus a countable set (cf. Remark \ref{rem:lattice}\eqref{item:countable}).
\item By  \eqref{direct-sum} and  \eqref{support-d}, $\supp (\mu) \subseteq V\oplus \Lambda=\cup_{a \in \Lambda} (V+a)$ which is a periodic superposition of parrallel affine subspaces $V+a$ of dimension $d_V<d$ (cf. Figure \ref{fig:SeveralSupports2bis}).
\item Each $\mu_a$ is supported by the affine space $V+a \approx \R^{d_V}$, and the theorem is written in such a way that some $\mu_a$ could be the zero measure. In particular, $\supp (\mu)$ may be just a subset of $V\oplus\Lambda$.
\item The triplet $(V,\Lambda,\{\mu_a\}_{a \in \Lambda})$ in the decomposition \eqref{gfnl} is however unique if we require also to have  
$
\overline{G(\supp (\mu))}=V \oplus \Lambda$ with $V \perp \Lambda
$
(cf. Corollary \ref{cor:uniq} in Appendix \ref{sec:uniq}).
\item By \eqref{levy-d}, the measure $\mu_0$ is symmetric and $\mu=\sum_{a \in \Lambda} \mu_a$ is as well although every other $\mu_a$ may not be symmetric (cf. Figure \ref{fig:SeveralSupports2bis}). 
\item By \eqref{levy-d-arxiv}, every $\mu_a$ are bounded except eventually $\mu_0$ which is the only measure whose support may contain the singularity $z=0$ (cf. Figure \ref{fig:SeveralSupports2bis}). 
\item For further representative examples of operators than the one in Figure \ref{fig:SeveralSupports2bis}, see Section \ref{sec:ex}.
\end{enumerate}
\end{remark}

\begin{figure}[h!]
	\centering
\includegraphics[width=0.65\textwidth]{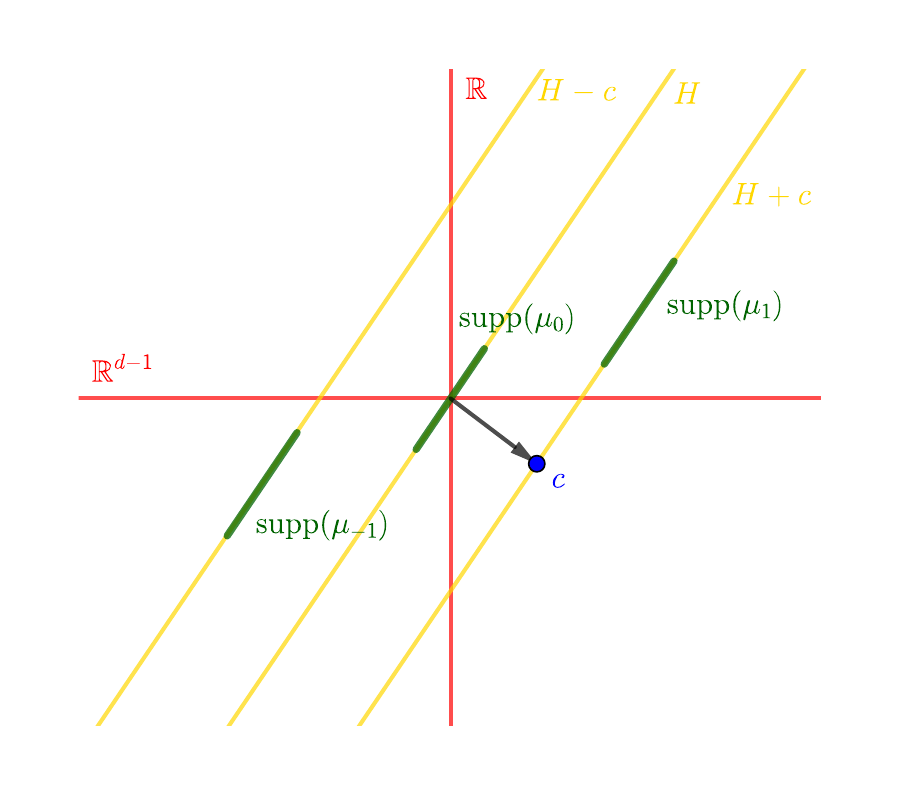}
\caption{ Example of decomposition of $\Levy^\mu$ as in \eqref{gfnl} with $V=H$ of codimension $1$, $\Lambda=c \Z$, and $\mu_{\pm 1}=\mu_{\pm c}$.} 
	\label{fig:SeveralSupports2bis}
\end{figure}

%

\begin{proof}[Proof of Theorem \ref{thm:consequences-d}]
Assume first that \eqref{gfnl} holds and let us prove \eqref{NL}. Recall that $\mu=\sum_{a \in \Lambda} \mu_a$ by \eqref{gfnl} and $\supp (\mu) \subseteq V \oplus \Lambda$ by \eqref{support-d}, where 
$
V \cap \textup{Span}_\R \Lambda=\{0\}
$
by \eqref{direct-sum}.
By Theorem \ref{thm:main-d}\eqref{item:d-t-bp}, it suffices to have $V \neq \R^d$ to get the negated version \eqref{NL} of the Liouville theorem, but this is assumed in  \eqref{direct-sum} which completes the proof of the  if  part.

Assume now that \eqref{NL} holds and let us prove that $\Levy^\mu$ is of the form \eqref{gfnl}. By Theorem \ref{decom:opt}, 
$$
\overline{G(\supp (\mu))}=
V\oplus\Lambda
$$ 
for some vector subspace $V \subseteq \R^d$ and relative lattice $\Lambda\subseteq \R^d$ such that 
\begin{equation}\label{lattice-detail}
V \cap \textup{Span}_\R \Lambda=\{0\}.
\end{equation}
By Theorem \ref{thm:main-d} and \eqref{NL}, $V \oplus \Lambda \neq \R^d$. It follows that $V \neq \R^d$ and  this  proves \eqref{direct-sum}. It only remains to construct $\mu_a$ satisfying \eqref{gfnl} and \eqref{levy-d}--\eqref{levy-d-arxiv}. We will need to show that for any $a,b \in \Lambda$ such that $a \neq b$,
\begin{equation}\label{lattice-detail-bis}
(V+a) \cap (V+b) = \emptyset.
\end{equation}
To prove \eqref{lattice-detail-bis}, assume that $x \in (V+a) \cap (V+b)$ and let us find a contradiction. In that case, $x=v+a=w+b$ with $v,w \in V$ and by \eqref{lattice-detail},
$$
v-w=a-b \in V \cap   \Lambda =\{0\}. 
$$  
This gives us the contradiction that $a=b$ and shows \eqref{lattice-detail-bis}. Now for any $B$,
\begin{equation*}
\begin{split}
\mu(B)
&=\mu \big(B \cap (V\oplus \Lambda)\big) \quad \mbox{(since $\supp (\mu) \subseteq V\oplus \Lambda$)}\\
& =\mu \Big(B \cap \big(\cup_{a \in \Lambda} (V+a)\big)\Big)\\
& = \sum_{a \in \Lambda} \underbrace{\mu \big(B \cap (V+a)\big)}_{:=\mu_a(B)}.
\end{split}
\end{equation*} 
The measures $\mu_a(\cdot)=\mu(\cdot \cap (V+a))$ satisfy \eqref{gfnl} and \eqref{support-d}. They are also nonnegative since $\mu \geq 0$, and by the symmetry of $\mu$,
\begin{equation*}
\begin{split}
\mu_a(B)&=\mu \big(B \cap (V+a)\big)=\mu \big(-(B \cap (V+a)) \big)\\
&=\mu \big( (-B) \cap (-V-a)\big)=\mu_{-a}(-B)
\end{split}
\end{equation*}
since $V=-V$. This proves 
\eqref{levy-d}. Using finally \eqref{as:mus},
\begin{equation}\label{very-last}
\sum_{a \in \Lambda}\int (|z|^2 \wedge 1) \dd \mu_a(z)=\int (|z|^2 \wedge 1) \dd \mu(z)<+\infty
\end{equation}
where for any $a \neq 0$,
$$
\int (|z|^2 \wedge 1) \dd \mu_a(z) \geq (\varepsilon_a^2 \wedge 1 \Big) \mu_k(\R^d)
$$
with the distance $\varepsilon_a:=\textup{dist}(0,V+a)$. The affine space $V+a$ does not contain $0$ if $a \neq 0$ which implies that $\varepsilon_a>0$. To be more precise, we have
$$
\varepsilon_a=|0-\textup{proj}_{V+a}(0)|=|\textup{proj}_{V+a}(0)|,
$$
with the orthogonal projection onto $V+a$. Or equivalently, 
$$
\varepsilon_a=|\textup{proj}_{V+a}(0)|=|a-\textup{proj}_{V}(a)|.
$$
If we can prove that 
\begin{equation}\label{detail-lattice-last}
\inf_{a \in \Lambda \setminus \{0\}} \varepsilon_a>0,
\end{equation}
we will obtain the last property \eqref{levy-d-arxiv} from \eqref{very-last} and complete the proof. 

 But \eqref{detail-lattice-last} is rather standard (cf. e.g. \cite{Mar03}). Let us give details for completeness.  We use that $\Lambda$ is discrete by Lemma \ref{lem:lattice} to infer that there is some $B_{\varepsilon}(0)$ with $\varepsilon>0$ such that 
$
B_{\varepsilon}(0) \cap \Lambda=\{0\}.
$
Let us now assume by contradiction that \eqref{detail-lattice-last} does not hold. This will imply the existence of some $\{a_n\}_{n} \subseteq \Lambda \cap B_{\varepsilon}(0)^c$ such that $|a_n-\textup{proj}_{V}(a_n)| \to 0$ as $n \to +\infty$. Note then that $b_n:=a_n/|a_n|$ converges to some $\overline{b} \neq 0$ along a subsequence which satisfies
$$
|b_n-\textup{proj}_{V}(b_n)| = \frac{|a_n-\textup{proj}_{V}(a_n)| }{|a_n|} \leq \frac{|a_n-\textup{proj}_{V}(a_n)| }{\varepsilon} \to 0
$$
as $n \to +\infty$. The limit then satisfies
$$
0\neq \overline{b}=\textup{proj}_{V}(\overline{b}) \in V \cap \textup{Span}_\R \Lambda =\{0\},$$ 
which is a contradiction with \eqref{lattice-detail} and completes the proof.
\end{proof}

\section{Representative examples of multi--$d$ operators }\label{sec:ex}

Let us illustrate Theorems \ref{thm:main-d} and \ref{thm:consequences-d} with model operators having or \textit{not} having the Liouville property. 


\subsection{Nonatomic measures and Liouville}
\label{sec:exunboundedMultiD} 

If $d=1$ the Liouville theorem can fail only for atomic measures (cf.  Proposition \ref{cor:consequences}).  In multi--$d$, this is slightly different.

\begin{figure}[h!]
	\centering
	\subfigure[$\supp(\mu)$ contains a ball. The Liouville property {\em holds}.]{\includegraphics[width=0.45\textwidth]{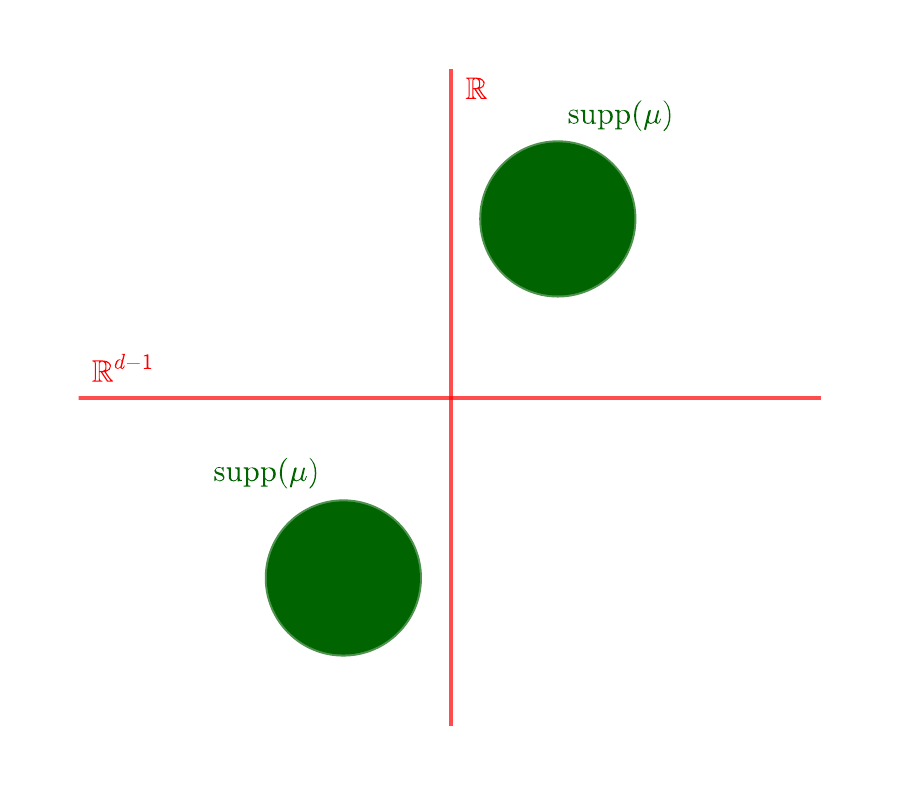}}
\hspace{0.25cm}
	\subfigure[$\supp(\mu)=\R^{d-1}\times\{0\}$. The Liouville property {\em fails}.]{\includegraphics[width=0.45\textwidth]{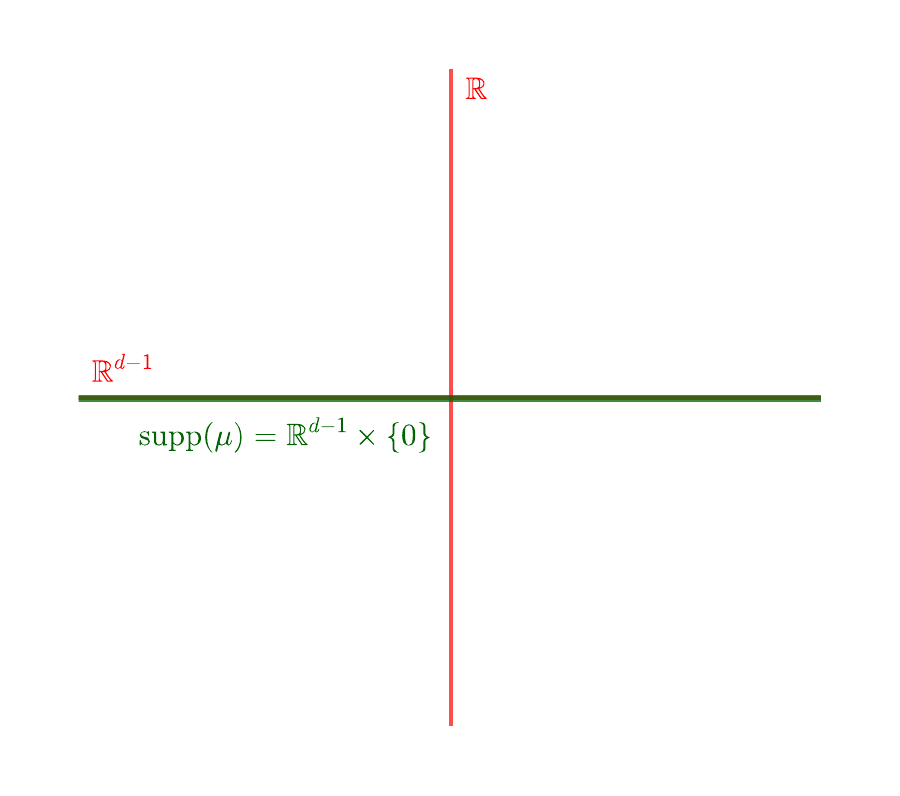}}
	\hspace{1cm}
	\subfigure[$\supp(\sigma)=\partial B_1(0)$. The Liouville property {\em holds}.]{\includegraphics[width=0.45\textwidth]{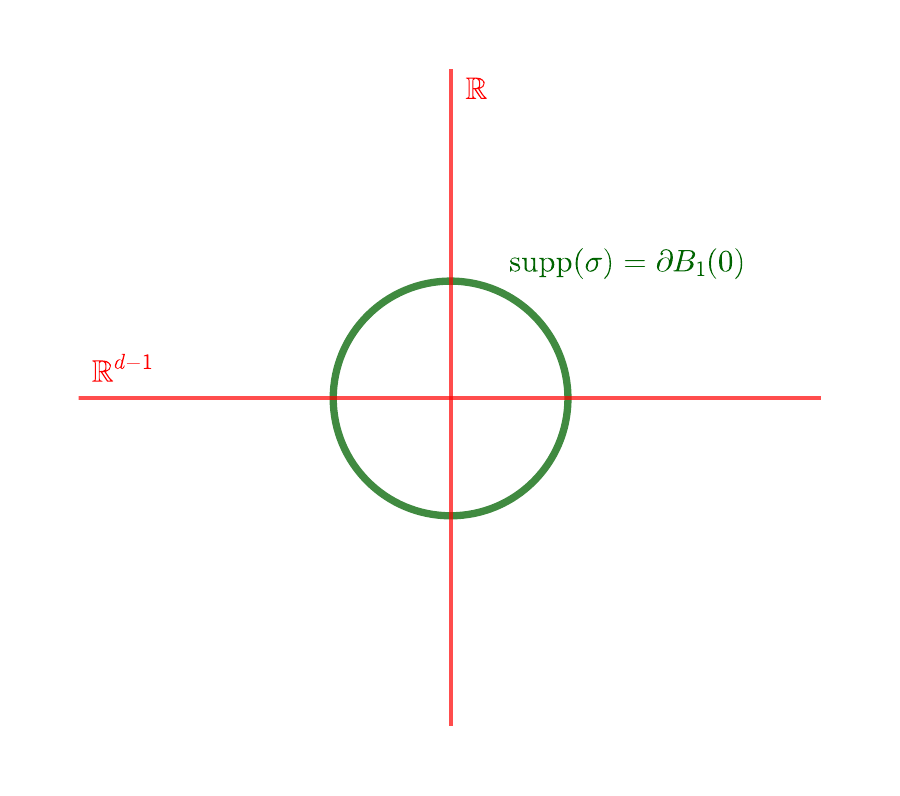}}
	\caption{ Supports satisfying or \textit{not} satisfying Thm. \ref{thm:main-d}\eqref{item:d-t-c}.}
	\label{fig:SeveralSupports}
\end{figure}


\begin{example}[Absolutely continuous measures satisfy Liouville]
If  the support of $\mu$ contains a ball  (cf. Figure \ref{fig:SeveralSupports}{\color{red}(a)}), then the Liouville theorem {\em holds} by Theorem \ref{thm:main-d}\eqref{item:d-t-c}. This is e.g. the case  for absolutely continuous measures like the fractional Laplacian, etc.
\end{example}

{ \begin{example}[Diffuse measures: Case 1]\label{ex:diff}
Consider now the non absolutely continuous diffuse measure associated to $\Levy^\mu=-(-\Delta_{d-1})^{\frac{\alpha}{2}}$ where $\Delta_{d-1}=\sum_{i=1}^{d-1} \partial_{x_i}^2$. It is of the form
\[ 
\dd \mu(z)= f(z) \dd z_1 \ldots \dd z_{d-1}\dd \delta_0(z_d) \quad  \text{with} \quad f(z)= c_{d,\alpha}\frac{\mathbf{1}_{\{z_d=0\}}}{|z|^{d-1+\alpha}}.
\]
Since $\supp (\mu)=H + c\Z$ with $H=\R^{d-1}$ and $c=0$, the Liouville theorem {\em fails} by Theorem \ref{thm:main-d}\eqref{item:d-t-c} (cf. Figure \ref{fig:SeveralSupports}{\color{red}(b)}). 
\end{example}

\begin{remark}[Unbounded measures at $z=0$]
In Example \ref{ex:diff}, $\mu$ is unbounded where as the Liouville theorem \textit{fails} for $\Levy^\mu$. This may occur only if $d>1$. Indeed,  by   Proposition \ref{cor:consequences},  every $1$--$d$ L\'evy  measure  for which the Liouville theorem fails should be purely atomic \textit{and} bounded up to $z=0$.
\end{remark}

\begin{example}[Diffuse measures: Case 2] 
\label{ex:surfaceMultiD} 
Let us consider another non absolutely continuous diffuse measure given by the mean value operator   
$$
\mathcal{M}[u](x)=\frac{1}{c_d} \int_{\partial B_1(0)} \big(u(x+z)-u(x)\big) \dd \sigma(z),
$$
 where $c_d$ is the area of $\partial B_1(0)$  (and $\dd \sigma(z)$ the surface measure supported by the unit ball).  The Liouville theorem {\em holds} by Theorem \ref{thm:main-d}\eqref{item:d-t-c} (cf. Figure \ref{fig:SeveralSupports}{\color{red}(c)}).
\end{example}

\begin{remark}[Classical Liouville]
If $\Delta u=0$ in $\R^d$ then it is well-known that $\mathcal{M}[u]=0$ in $\R^d$ with the operator from Example \ref{ex:surfaceMultiD}. In particular, the Liouville theorem for $\mathcal{M}$ implies the classical Liouville theorem for $\Delta$.
\end{remark}

\subsection{Discrete measures and the Kronecker theorem}
Let us  consider atomic measures with the minimal possible number of points in the support to have Liouville. By parts \eqref{item:d-t-bp} and \eqref{item:d-t-c} of Theorem \ref{thm:main-d}, we have to avoid the cases where 
$$
{{\overline{G(\supp (\mu))} = \Lambda}} \quad \mbox{or} \quad \overline{G(\supp (\mu))}  \subseteq H,
$$ 
for some (full) lattice $\Lambda$ of $\R^d$ or some vector space $H$ of codimension $1$. This typically occurs for  
$$
\mu(z)=\sum_{n=1}^N  (\delta_{a_n}(z)+ \delta_{-a_n}(z)),
$$
where $(a_1,\dots,a_{N=d})$ is a basis of $\R^d$ or  $\textup{Span}_\R \{a_1,\dots,a_N\} \neq \R^d$.  The minimal situation to consider to have Liouville is thus when $\supp (\mu)$ contains at least $d+1$ points $a_1,\dots,a_{d+1}$ such that 
$$
\textup{Span}_{\R}\{a_1,\dots,a_{d+1}\}=\R^d. 
$$
By changing coordinates, it is sufficient to understand this situation for 
\begin{equation}\label{Kronecker-meas}
\mu(z)=\delta_c(z) +\delta_{-c}(z)  +\sum_{i=1}^d  (\delta_{e_i}(z)+ \delta_{-e_i}(z)),
\end{equation}
where $c \neq0$ and $(e_1,\dots,e_d)$ is the canonical basis. In that case
$$
G(\supp (\mu))=c \Z+\Z^d, 
$$
and the Kronecker theorem (cf. e.g. page 507 in  \cite{GoMo16})) provides a necessary and sufficient condition to have $\overline{c \Z+\Z^d}=\R^d$. Let us recall this result. 

\begin{theorem}[Kronecker's approximation]\label{thm:CharKron}
For any $c=(c_1,\dots,c_d)\in \R^d$, we have $\overline{c \Z+\Z^d}=\R^d$ if and only if $\{1,c_1,\dots,c_d\}$ is linearly independent over $\Q$. 
\end{theorem}

\begin{corollary}\label{coro:CharKron}
The measure $\mu$ in \eqref{Kronecker-meas} satisfies the Liouville property if and only if $\{1,c_1,\dots,c_d\}$ is linearly independent over $\Q$.
\end{corollary}

\begin{proof}
Use Theorems \ref{thm:main-d} and \ref{thm:CharKron}.
\end{proof}

\begin{remark}
\begin{enumerate}[(a)]
\item In $1$--$d$, we recover that $\mu=\sum_{\pm} (  \omega_1  \delta_{\pm a}+  \omega_2  \delta_{\pm b})$ satisfies the Liouville property if and only if $\frac{b}{a} \notin \Q$.
\item   The condition that $\{1,c_1,\dots,c_d\}$ is linearly independent over $\Q$  is  related  to Diophantine equations (cf. \cite{GoMo16}).   
\end{enumerate}
\end{remark}


\subsection{Finite difference operators} 
\label{sec:exfinitediffMultiD}
 Let us now give some examples coming from numerical methods.
\begin{example}
Let $(e_1,\ldots,e_d)$ be the  canonical  basis in $\R^d$ and $h>0$. Consider the discrete Laplacian 
\begin{equation*}
\Delta_h u(x)= \sum_{i=1}^d \frac{u(x+h e_i)+u(x-h e_i)-2u(x)}{h^2},
\end{equation*}
where $\mu(z)= \sum_{i=1}^d (\delta_{h e_i}(z)+\delta_{-he_i}(z))h^{-2}$. By Theorem \ref{thm:main-d}\eqref{item:d-t-bp}, the Liouville theorem {\em fails} since $G(\supp (\mu))=h\Z^d$ is a lattice.
\end{example}

To get  discrete  operators satisfying Liouville, we need more nodes and they must be  defined on  nonuniform  grids  as in the following example:

\begin{example}\label{ex:discLapMultiDbis}
 Let $h,\rho>0$.  The following operator is a nonstandard discretization of the Laplacian 
$$
 \Levy_h[u](x):=\frac{1}{2}\Delta_h u(x)+ \frac{1}{2}\Delta_{h \rho} u(x) 
$$
with measure 
\[
\mu(z)= \sum_{i=1}^d \frac{1}{2h^2} \big(\delta_{he_i}(z)+\delta_{-he_i}(z)\big) +\sum_{i=1}^d \frac{1}{2h^2 \rho^2} \big(\delta_{h \rho e_i}(z)+\delta_{-h \rho e_i}(z)\big). \]
By Lemma \ref{lem:dense},
$
G(\supp (\mu))=h (\Z+\rho\Z)^d$ is dense
in $\R^d$ if and only if $\rho \notin\Q$. This is also a necessary and sufficient condition on $\Levy_h$ to satisfy Liouville (cf. Figure \ref{fig:SeveralSupports2}). 
\end{example}

\begin{figure}[h!]
	\centering
\includegraphics[width=0.65\textwidth]{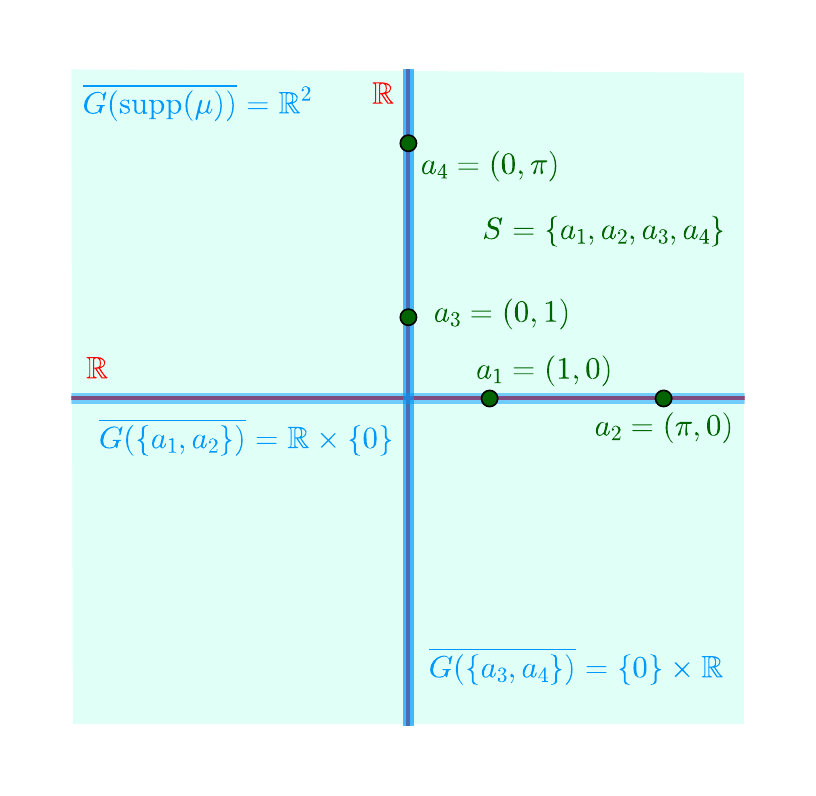}
\caption{ Graphical representation of $\supp (\mu)$ for $\mathcal{\Levy}_h$ in Example \ref{ex:discLapMultiDbis} with $d=2$,  $h=1$, and $\rho=\pi$.  We have $\supp (\mu)=S \cup (-S)$ and $\overline{G(\supp(\mu))}=\R \times \R$. }
	\label{fig:SeveralSupports2}
\end{figure}

\section*{Acknowledgements}

J. Endal and E. R. Jakobsen were supported by the Toppforsk (research excellence) project Waves and Nonlinear Phenomena (WaNP), grant no. 250070 from the Research Council of Norway. 
F. del Teso and J. Endal are grateful to  Laboratoire de Math\'ematiques de Besan\c{c}on  (LMB, UBFC) and   Ecole Nationale Sup\'erieure de M\'ecanique et des Microtechniques (ENSMM)  for providing accommodation during their visit in May in 2018.


\appendix

\section{Uniqueness of the decomposition $\mu=\sum \mu_a$}\label{sec:uniq}

In this appendix, we discuss the uniqueness of the decomposition \eqref{gfnl} in Theorem \ref{thm:consequences-d}. We will need the elementary lemma below for the uniqueness of the decomposition in Theorem \ref{decom:opt}.

\begin{lemma}\label{lem:uniq}
A closed subgroup $G$ of $\R^d$ can be written in a unique way in the form $G=V \oplus \Lambda$ for some vector space $V \subseteq \R^d$ and relative lattice $\Lambda \subseteq \R^d$ such that $V \perp \Lambda$. We then write $G=V \oplus_\perp \Lambda$.
\end{lemma}

\begin{remark}
For an illustrative example, see Figure \ref{fig:OrtoProjj}.
\end{remark}

\begin{figure}[h!]
	\centering
	\includegraphics[width=0.65\textwidth]{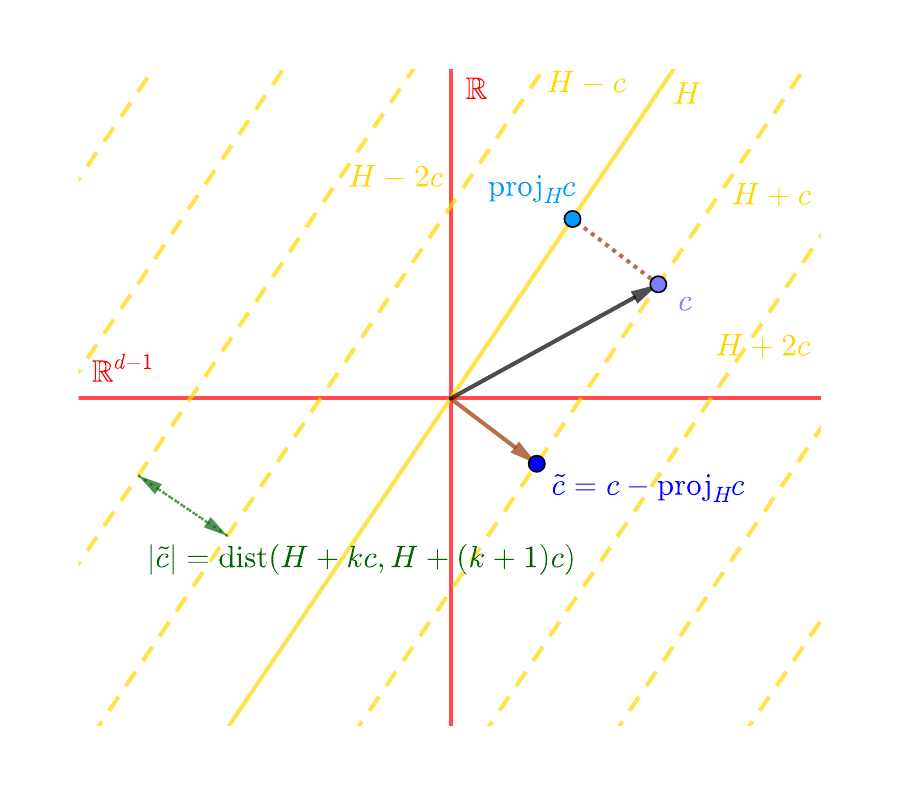}
	\caption{Graphical representation of $G=V \oplus \Lambda =V \oplus_\perp \tilde{\Lambda}$ with $V=H$ of codimension $1$, $\Lambda =c \Z$, {{and  $\tilde{\Lambda}=\tilde{c} \Z $.}}}
	\label{fig:OrtoProjj}
\end{figure}

\begin{proof}[Proof of Lemma \ref{lem:uniq}]
Let us first show the existence. By Theorem \ref{decom:opt}, $G=V \oplus \Lambda$ with $V$ and $\Lambda$ not necessarily orthogonal. Take then
$$
\tilde{\Lambda}:=\{a-\textup{proj}_V(a):a \in \Lambda\}
$$
with the orthogonal projection onto $V$, and note that 
 $
V\oplus\Lambda=V\oplus_\perp\tilde{\Lambda}
$.  This completes the proof of the existence.

Let us now prove the uniqueness. Assume that $G=V\oplus_\perp\Lambda=\tilde{V}\oplus_\perp\tilde{\Lambda}$ for vectors spaces $V,\tilde{V}$ and relative lattices $\Lambda,\tilde{\Lambda}$.  Since lattices are discrete by Lemma \ref{lem:lattice}, there exists a ball $B_{\varepsilon}(0)$ of $\R^d$ such that
\begin{equation}\label{uniq-detail1}
\Lambda \cap B_{\varepsilon}(0)=\tilde{\Lambda} \cap B_{\varepsilon}(0)=\{0\}.
\end{equation} 
We then claim that
\begin{equation}\label{uniq-detail2}
G \cap B_{\varepsilon}(0)=V \cap B_{\varepsilon}(0)=\tilde{V} \cap B_{\varepsilon}(0).
\end{equation}
Indeed, if $x\in G \cap B_{\varepsilon}(0)$ then $x=v+a$ for some $v \in V$ and $a \in \Lambda$ such that $|v+a| < \varepsilon$ and $v \perp a$. In particular, $|v+a|=\sqrt{|v|^2+|a|^2}$ and it follows that $|a| <\varepsilon$. By \eqref{uniq-detail1}, we deduce that $a=0$ and thus $x = v \in V$ with $|v| <\varepsilon$. This shows that 
$$
G \cap B_{\varepsilon}(0) \subseteq V \cap B_{\varepsilon}(0),
$$ 
and since the reverse inclusion is obvious we deduce the first part of \eqref{uniq-detail2}. We argue similarly for the part in $\tilde{V}$ and complete the proof of our claim \eqref{uniq-detail2}. 
But $V=\textup{Span}_\R (V \cap B_{\varepsilon}(0))$ and $\tilde{V}=\textup{Span}_\R (\tilde{V} \cap B_{\varepsilon}(0))$, and therefore  
$$
V=\tilde{V}.
$$
If now $a \in \Lambda$ then $a=v+\tilde{a}$ with $v \in V$ and $\tilde{a} \in \tilde{\Lambda}$. It follows that
$
v=a-\tilde{a} \perp v
$
and the only possibility is that $v=0$. Hence $a = \tilde{a} \in \tilde{\Lambda}$ and we conclude that $\Lambda \subseteq \tilde{\Lambda}$. We show the reverse inclusion in the same way and conclude the proof.
\end{proof}

Here is now a general decomposition result for Radon measures.

\begin{lemma}
Let $\mu$ be a nonnegative Radon measure on $\R^d \setminus \{0\}$. Then there exists a unique triplet $(V,\Lambda,\{\mu_a\}_{a \in \Lambda})$ such that 
\begin{enumerate}
\item[$\bullet$] $V \subseteq \R^d$ is a vector space,
\item[$\bullet$] $\Lambda \subseteq \R^d$ is a relative lattice,
\item[$\bullet$] $\mu_a$ is a nonnegative Radon measure on $\R^d \setminus \{0\}$ for any $a \in \Lambda$, 
\end{enumerate}
and the following properties hold:   
\begin{equation*}
 \overline{G(\supp(\mu))}=V \oplus_\perp \Lambda  \quad \mbox{and} \quad \mu=\sum_{a \in \Lambda} \mu_a \quad \mbox{with} \quad \supp (\mu_a) \subseteq V+a.
\end{equation*}
\end{lemma}

\begin{proof}
Since $\overline{G(\supp (\mu))}$ is closed,  Lemma \ref{lem:uniq} implies that  $$
\overline{G(\supp(\mu))}=V \oplus_\perp \Lambda
$$
where such a decomposition is unique. If $V=\R^d$ then $\Lambda=\{0\}$ and the only family $\{\mu_a\}_{a \in \Lambda}$ that will satisfy the theorem is given by the singleton $\{\mu_{a}=\mu\}_{a \in \{0\}}$. If now $V \neq \R^d$ then $V \oplus_\perp \Lambda=\cup_{a \in \Lambda} (V+a)$ where the  union  is disjoint, and we can argue as in the proof of Theorem \ref{thm:consequences-d} to show that $\mu=\sum_{a \in \Lambda} \mu_a$ with the  trace  measures $\mu_{a}(\cdot)=\mu(\cdot \cap (V+a))$ satisfying 
$$
\supp (\mu_a) \subseteq V+a.
$$
Moreover, if $\mu =\sum_{a \in \Lambda} \tilde{\mu}_a$ for other Radon measures satisfying $\supp (\tilde{\mu}_a) \subseteq V+a$ then for any fixed $a_0 \in \Lambda$ and all Borel sets $B \subseteq V+a_0$,
$$
\mu(B)=\sum_{a \in \Lambda} \mu_a(B)=\mu_{a_0}(B) 
$$ 
and similarly $\mu(B)=\tilde{\mu}_{a_0}(B)$. This shows that $\mu_{a_0}=\tilde{\mu}_{a_0}$ for any such arbitrarily given $a_0 \in \Lambda$ and completes the proof.
\end{proof}

 As an immediate consequence, we get the following  uniqueness result on the decomposition \eqref{gfnl} in Theorem \ref{thm:consequences-d}.

\begin{corollary}\label{cor:uniq}
Assume \eqref{as:mus} and the negated version \eqref{NL} of the Liouville theorem. Then there exists a unique triplet $(V,\Lambda,\{\mu_a\}_{a \in \Lambda})$  satisfying \eqref{gfnl}--\eqref{levy-d-arxiv} and 
$$
\overline{G(\supp(\mu))}=V \oplus_\perp \Lambda.
$$ 
\end{corollary}


\section{Proof of Lemma \ref{lem:dense}}\label{proof:density}

The proof uses elementary arguments (cf. e.g. \cite{God87}). Let us give the details for completeness.

\begin{lemma}\label{lem:DistanceArbitrarySmall}
Let $a,b>0$ and define
$
 G^* :=\{nb-\lfloor n\frac{b}{a}\rfloor a \,:\, n\in\N_+\}. 
$
If $\frac{b}{a}\not\in\Q$, then $\inf  G^* =0$.
\end{lemma}


\begin{proof}[Proof of Lemma \ref{lem:DistanceArbitrarySmall}]

First note that $ G^* \subseteq (0,a)$:  Since $\frac{b}{a}\not\in\Q$, $\floor{n\frac{b}{a}}<n\frac{b}{a}<\floor{n\frac{b}{a}}+1$ and thus $\floor{n\frac{b}{a}}a<nb<\floor{n\frac{b}{a}}a+a$. Hence
\[
0<nb-\floor{n\frac{b}{a}}a<a.
\]

To prove the result we assume by contradiction that $\inf  G^* =l>0$ for some $l\in(0,a)$. Since $l$ is strictly positive, we can find $N\in \N_+$ such that 
\[
 Nl<a  \leq (N+1)l .
\]
Define $r:= (N+1)l-a$ and note that
\begin{equation}\label{detail-sp}
r= Nl-a +l < l  < a.
\end{equation}
By definition of infimum, there exist sequences  $\{\varepsilon_k\}_{k}\subset[0,1)$ and $\{n_k\}_{k}\subseteq \N_+$  with $\varepsilon_k\to0$  as $k\to +\infty$ such that 
\begin{equation*}
n_k b- \floor{n_k\frac{b}{a}}a=l+\varepsilon_k.
\end{equation*}
 If $\varepsilon_k=0$ then $\frac{l}{a} \notin \Q$ and therefore $r>0$. This proves that 
\begin{equation*}
\text{either} \quad r>0 \quad \mbox{or} \quad \{\varepsilon_k\}_{k} \subset(0,1).
\end{equation*}
Using then that 
\begin{equation}\label{eq:est1}
\begin{split}
(N+1)n_k b&= (N+1) \floor{n_k\frac{b}{a}}a+(N+1)l+(N+1)\varepsilon_k\\
&=(N+1)\floor{n_k\frac{b}{a}}a+a+r+(N+1)\varepsilon_k,\\
\end{split}
\end{equation}
it follows that
$((N+1)\floor{n_k\frac{b}{a}}+1)a< (N+1)n_kb$, 
and hence,
\begin{equation}\label{detail-e-1}
(N+1)\floor{n_k\frac{b}{a}}+1< (N+1)n_k\frac{b}{a}.
\end{equation}
Moreover, since  $r<a$ by \eqref{detail-sp}, we have $r+(N+1)\varepsilon_k<a$ for $k$ large enough,  which by \eqref{eq:est1} again leads to $(N+1)n_k b<((N+1)\floor{n_k\frac{b}{a}}+1)a+a$
and then
\begin{equation}\label{detail-e-2}
(N+1)n_k \frac{b}{a}<((N+1)\floor{n_k\frac{b}{a}}+1)+1.
\end{equation}
 By \eqref{detail-e-1} and \eqref{detail-e-2},  we conclude that 
$$
(N+1)\floor{n_k\frac{b}{a}}+1= \floor{(N+1)n_k\frac{ b}{a}}
$$
 for all sufficiently large $k$. 
Then since $(N+1)n_k b-\floor{(N+1)n_k \frac{b}{a}}a\in  G^* $, we can use \eqref{eq:est1} with $\varepsilon_k<\frac{l-r}{N+1}$  (remember that $l-r>0$ by \eqref{detail-sp})  to get 
\[
(N+1)n_k b-\floor{(N+1)n_k \frac{b}{a}}a=r+(N+1)\varepsilon_k< r+(N+1)\frac{l-r}{N+1}=l.
\]
This contradicts the fact that $\inf  G^* =l$ and concludes the proof.
\end{proof}

\begin{proof}[Proof of Lemma \ref{lem:dense}]
 The only if part can be shown from e.g. Remark \ref{rem:dense-discrete}, but we omit the detail since it was not used during the proofs of our main results. Let us focus on the if part. 
 It suffices to prove it for $a,b >0$, which we assume together with $\frac{b}{a} \notin \Q$.   Fix any  $z \in [0,+\infty)$  and $\varepsilon>0$. Consider the set 
$$
{ G^* }= \Big\{nb-\floor{n\frac{b}{a}}a: \  n\in \N_+\Big\}
$$
as defined in Lemma \ref{lem:DistanceArbitrarySmall} 
and  recall  that $ G^* \subseteq (a \Z+b \Z)\cap (0,a)$  (cf. the proof of Lemma \ref{lem:DistanceArbitrarySmall}).  By Lemma \ref{lem:DistanceArbitrarySmall}, there is $ z_0  \in  G^*$ such that $0< z_0 <\varepsilon$. We then take  $N \in \N_+$  such that 
\[
 N z_0 < z  \leq  (N+1)z_0, 
\] 
and observe that
$
 |z-N z_0|< z_0<\varepsilon. 
$
Since $ z_0 =nb-\floor{n\frac{b}{a}}a$ for some $n\in\N_+$,  $z_0$  obviously belong to  $a \Z+b \Z$. We can do the same reasoning if $z \in (-\infty,0]$  and we conclude that  $a \Z+b \Z$  is dense in $\R$.
\end{proof}

%
%
%
%
%
%
%
%

\end{document}